\documentclass[final]{siamltex}

\usepackage{inputenc}
\usepackage[english]{babel}
\usepackage{graphicx, subfigure}
\usepackage{amsmath}
\usepackage{color}
\usepackage{amssymb}
\usepackage{cite}
\usepackage{mdwmath}
\usepackage{multirow}
\usepackage{mdwtab}
\usepackage{verbatim}
\usepackage{enumerate}
\usepackage{amsfonts}
\usepackage{hyperref}

\newcommand{\FFF}{{\mathcal{F} }} 
\newcommand{\VVV}{{\mathcal{V} }} 
\newcommand{\RRR}{{\mathrm I\! \textsc{R} }}

\newcommand{\ZZZ}{{\mathbb{Z} }}

\newcommand{\AAA}{{\mathcal{A} }}

\newcommand{\HH}{{\mathcal{H} }}
\newcommand{\supeps} {{x_\epsilon,p_\epsilon,s_\epsilon,\xi_\epsilon,\pi_\epsilon,\varsigma_\epsilon}}
\newcommand{\vv}[3]{{\sideset{_{}^{#2}}{_{#1}^{#3}}{\mathop{v}}}}
\newcommand{\closure}[2][3]{{}\mkern#1mu\overline{\mkern-#1mu#2}} 

\newtheorem{remark}{\textsc{ {Remark}}}[section]

\title{Level-set approach for Reachability Analysis of Hybrid Systems under Lag Constraints\thanks{This 
        work was co-funded by Renault SAS under grant ANRT CIFRE n° 928/2009.}}

\author{G. Granato\thanks{Renault SAS, Advanced Electronics Division, 
        TCR RUC T 65, 78286 Guyancourt Cedex, France
        ({\tt giovann.granato@renault.com}). \'Ecole Nationale Sup\'erieure de Techniques Avanc\'ees,
        Unit\'e de Math\'ematiques Appliqu\'ees, 
        32, boulevard Victor, 75015, Parix Cedex 15, France 
       ({\tt giovanni.granato@ensta-paristech.fr}) }
        \and H. Zidani\thanks{\'Ecole Nationale Sup\'erieure de Techniques Avanc\'ees,
        Unit\'e de Math\'ematiques Appliqu\'ees, 
        32, boulevard Victor, 75015, Parix Cedex 15, France 
       ({\tt hasnaa.zidani@ensta-paristech.fr}) }}

\begin{document}
	
	\maketitle
	
	\begin{abstract}
		This study aims at characterizing a reachable set of a hybrid dynamical system with a lag constraint in the switch control. 
		The setting does not consider any controllability assumptions and uses a level-set approach. 
		The approach consists in the introduction of on adequate hybrid optimal control problem with lag constraints on the switch control whose value function allows a characterization of the reachable set.
		The value function is in turn characterized by a system of quasi-variational inequalities (SQVI).
		We prove a comparison principle for the SQVI which shows uniqueness of its solution.
		A class of numerical finite differences schemes for solving the system of inequalities is proposed and the convergence of the numerical solution towards the value function is studied using the comparison principle.
		Some numerical examples illustrating the method are presented.
		Our study is motivated by an industrial application, namely, that of range extender electric vehicles. 
		This class of electric vehicles uses an additional module -- the range extender -- as an extra source of energy in addition to its main source -- a high voltage battery.
		The reachability study of this system is used to establish the maximum range of a simple vehicle model.
	\end{abstract}
	
	

	\pagestyle{myheadings}
	\thispagestyle{plain}
	\markboth{G. GRANATO AND H. ZIDANI}{REACHABILITY ANALYSIS OF HYBRID SYSTEMS}

	\section{Introduction}

	This paper deals with the characterization of a reachable set of a hybrid dynamical system with a lag constraint in the switch control. 
	The approach consists in the introduction of on adequate hybrid optimal control problem with lag constraints on the switch control whose value function allows a characterization of the reachable set.	
	
	The term hybrid system refers to a general framework that can be used to model a large class of systems.
	Broadly speaking, they arise whenever a collection continuous- and discrete-time dynamics are put together in a single model.
	In that sense, the discrete dynamics may dictate switching between the continuous dynamics, jumps in the system trajectory or both.
	Moreover, they can contain specificities, as for instance, autonomous jumps and/or switches, time delay between discrete decisions, switching/jumping costs.
	This work considers a particular class of hybrid system where only switching between continuous dynamics are operated by the discrete logic, with no jumps in the trajectory, and there are no switching costs.
	In addition, switch decisions are constrained to be separated in time by a non-zero interval, fact which is referred to as switching lag.
	
	Before referring to the reachability problem in the hybrid setting, the main ideas are introduced in the non-hybrid framework.
	Given a time $t>0$, a closed target set $X_0$ and a closed admissible set $K$, considering a controlled dynamical system
	 \begin{equation} \label{eq.classique}
		\dot y(\tau) = f(\tau,y(\tau),u(\tau)), ~~ \text{a.e. } \tau \in [0,t], 
	\end{equation}•
	where $f: \RRR^+ \times \RRR^d \times U \to \RRR^d$ and $u: \RRR^+ \rightarrow U $ is a measurable function, the reachable set $R_{X_0}$ at time $t$ is defined as the set of all initial states $x$ for which there exists a trajectory that stays inside $K$ on $[0,t]$ and arrives at the target:
	\begin{equation*}
		R_{X_0}(t) := \{ x ~|~  (y,u) \text{ satisfies \eqref{eq.classique}}, \mbox{ with } y(0)=x \text{ and } y(t) \in X_0 \text{ and } y(s)\in K \mbox{ on } [0,t] \}.
	\end{equation*}•
	It is a known fact that the reachable set can be characterized by the the negative region of the value function of an optimal control problem. 
	For this, following the idea introduced by Osher \cite{OsherSethian88}, one can consider the control problem defined by:
	\begin{equation} \label{eq.pbclassique}
		v(x,t):=\inf\{v_0(y(t))\mid (y,u) \text{ satisfies \eqref{eq.classique}}, \mbox{ with } y(0)=x
				 \mbox{ and }  y(s)\in K \mbox{ on } [0,t]\},
	\end{equation}	
	where $v_0$ is a Lipschitz continuous function satisfying $v_0(x)\leq 0 \Longleftrightarrow x\in X_0$ (for instance,$v_0$ can be the signed distance $d_{X_0}$ to $X_0$).
	Under classical assumptions on the vector-field $f$, one can prove that the reachable set is given by
	$$ R_{X_0}(t)=\{x\in K, v(x,t)\leq 0 \}.$$
	Moreover, when $K$ is equal to $\RRR^d$, the value function has been shown to be the unique viscosity solution
of a Hamilton-Jacobi-Bellman (HJB) equation \cite{Bardi97optimalcontrol}:
	$$ \partial_t v+\sup_{u\in U}(-f(s,x,u)\cdot \nabla v)=0 \quad \mbox{ on } \RRR^d \times (0,t],$$
	with the initial condition $v(x,0)=v_0(x)$.
 
 	When the set $K$ is a subset of $\RRR^d$ ($K\neq \RRR^d$), the characterization of $v$ by means of a HJB equation becomes a more delicate matter and usually requires some additional controllability properties \cite{CardQuincampoix, BFZ}.
 	However, it was pointed out in \cite{BFZ} that in case of state-constraints, the auxiliary control problem should be introduced as
	\begin{equation}\label{eq.pbcontraint}
	 V(x,t):=\inf\left\{v_0(y(t)) \bigvee \max_{\theta\in[0,t]}g(y(\theta))
		\mid (y,u) \mbox{ satisfies \eqref{eq.classique}}, \mbox{ with } y(0)=x \right\},
	\end{equation}	
	where $g$ is any Lipschitz continuous function satisfying $g(x)\leq 0 \Longleftrightarrow x\in K$ (again, $g$ can be the signed distance $d_{K}$ to $K$).
	This new control problem involves a supremum cost but does not include any state constraints. 
	The reachability set is still given by
	$$ R_{X_0}(t)=\{x\in K, V(x,t)\leq 0 \},$$
	and the value function $V$ is the unique viscosity solution of a variational inequality
	$$\min( \partial_t V+\sup_{u\in U}(-f(s,x,u)\cdot \nabla V), V-g)=0 \quad \mbox{ on } (0,t]\times \RRR^d,$$
	with the initial condition $v(x,0)=\min(v_0(x),g(x))$, but no controllability assumption is needed.	
	
In this paper, we are interested in the extension of the reachability framework to some class of control problems of Hybrid systems.

	Let us recall that a hybrid dynamical system is a collection of controlled continuous-time processes selected through a high-level discrete control logic.
	A general framework for the (optimal) control hybrid dynamical systems was introduced in \cite{Branicky}.
	Several papers deal with the optimal control problem of hybrid systems, let us just mention here the papers
of \cite{Garavello, Arutyunov,Vinter-Caines,Sussman} where the optimality conditions in the form of Pontryagin\rq{} principle are studied and \cite{Zhang, Dolcetta:143, Dharmatti2} where the HJB approach is analyzed.
	
	A feature of the hybrid system used in our work is a time lag between two consecutive switching decisions.
	From the mathematical viewpoint this remove the particularities linked to Zeno-like phenomena \cite{zenoYuBarbot}.
	Indeed, the collection of state spaces is divided in subsets labeled in three categories according to whether they characterize discrete decisions as optional, required (autonomous) or forbidden.
	Landing conditions ensure that whatever the region of the state space the state vector \lq\lq{}lands\rq\rq{} after a switch no other switch is possible by requiring a positive distance (in the Hausdorff sense) between the landing sets and the optional/autonomous switch sets.	
	In the other hand, when allowed to switch freely without any costs, when no time interval is imposed between discrete transitions, a controller with a possibly infinite number of instantaneous switches may become admissible.
	Switch costs can be introduced in order to rule out this kind of strategy by the controller as it becomes over-expensive to switch to a particular mode using superfluous transitions.
	However, such costs do not make sense in the level-set approach used in this paper.
	
	Our study is motivated by an industrial application, namely, that of range extender electric vehicles. 
	This class of electric vehicles uses an additional module -- the range extender -- as an extra source of energy in addition to its main energy source (a high voltage battery).
	For that matter, an adequate class of hybrid systems considers only mode switching, i.e. switching between different continuous time dynamics, without any trajectory discontinuities.
	Moreover, no transition costs are taken into account and the switching decisions can be done freely without any penalty.
	Also motivated by the application, the discrete control must respect a time interval of at least $\delta >0$ between two consecutive switching decisions.
	This decision lag condition can be viewed as replacing landing conditions and positive switch costs requirements \cite{Dharmatti2}.
	
	Diffusion processes with impulse controls including switch lags are studied in \cite{pham}, where it is considered the idea of introducing a state variable to keep track of the time since the last discrete control decision. 
	There, in addition, discrete decisions also suffer from a time delay before they can manifest in the continuous-time process.
	In that case, one has the possibility of scheduling discrete orders whenever the time for a decision to take place may be longer than that of deciding again.
	Then, the analysis also includes keeping record of the nature of this scheduled orders.
	This work inspired the idea of a state variable locking possible transitions used here.

	To study the reachability sets for our system, we follow the level set approach and adapt the ideas developed in \cite{BFZ} to hybrid systems by proposing a suitable control problem which allows us to handle in a convenient way the state constraints and the decision lag. 
	It is proven that 
	\begin{equation*}
			R_{X^0}(s) =  \{ x ~|~ \exists (q,p) \in Q\times P,  ~v(x,q,p,s) \leq 0 \}.
	\end{equation*}
	where $v$ denotes the value function of a hybrid optimal control problem.
	Thus, through a characterization of $v$, one obtains the desired reachable set, defined in the hybrid context, $x$ is in the (physical) state of the system, $q$ is the discrete
variable and $p$ is a switch lock variable (defined further below).
	 Here the main difficulty is to characterize the value function associated to the control problem. 
	 It turns out that this value function satisfies a quasi-variational HJB  inequality system (in the viscosity sense)
	\begin{eqnarray}
		\min(\partial_s u + \partial_p u + H(s,x,q,\nabla_x u)  , u-\varphi(x) ) = 0 &,&  ~(x,q,p,s) \in \Omega, \label{HJBinside}\\
		u(x,q,p,s) -(Mu)(x,q,p,s)= 0 &,&~ p=0 \\			
		u(x,q,p,0) = \max(\phi(x),\varphi(x)).&&
	\label{HJBmain}
	\end{eqnarray}
	where $t>0$, $d>0$, $X = \RRR^d$, $Q = \{0,1\}$, $P=(0,t]$, $T=(0,t]$, $\Omega = X \times Q \times P \times T$ and where $\phi,\varphi : X \to \RRR$ are target and obstacle indicator functions (defined properly further below).
	Here, $M$ is a non-local switch operator that acts whenever the state variable $p$ touches the boundary $p=0$.
	Moreover, we give a comparison principle of this system. 
	Usually, the proof of the comparison principle requires some transversality assumptions that do not make sense in the kind of applications we are interested in. 
	However, while the decision lag complicates the structure of the problem it also plays a role in the proof of comparison principle (the same role that the transversality assumed in \cite{Zhang,Dharmatti}) acting as a kind of landing condition in the sense explained above. 
	The proof of the comparison principle is close to the one given in \cite{Dharmatti2,BarlesSheetal} and adapts the idea of using \lq\lq{}friendly giants\rq\rq{}-type functions. 

	This work is motivated by an application in the automobile industry, namely, calculating the autonomy of a class of electric vehicles (EVs),  the range extender electric vehicle (REEV) class, which posses two distinct sources of energy from which we can use to tract the vehicle.
	In this setting, the study aims at finding the control sequence of the two energy sources that allows the vehicle the reach the furthest possible point of a given route.
	The REEV is modeled as a hybrid dynamical system in which the state vector represents the energy capacities of the two different energy sources embedded in the vehicle.	
		
	This paper is organized as follows: firstly, it describes the industrial application motivating the study and states the associated hybrid optimal control problem (in a slightly more general setting than that required for the application). 
	Then, the reachable set and the value function are defined and a dynamic programming principle for the value function is obtained.
	The value function is shown to be Lipschitz continuous and also a solution of an system of quasi-variational inequalities (SQVI).
	It follows with the proof of a comparison principle for the SQVI that ensures uniqueness of its solution and is used to show the convergence of a class of numerical schemes for the computation of the value function.
	Lastly results of some numerical simulations evaluating the autonomy of a REEV toy model and illustrating the convergence of a discretization scheme are presented.

\section{Motivation and Problem Settings}

	\subsection{Range Extender Electric Vehicles}

		A range extender (RE)  electric vehicle is an electric vehicle that disposes of an additional source of energy besides the main high voltage (HV) traction battery.
		Both the vehicle's energy sources are considered to have normalized energetic capacities -- thus valued between $0$ and $1$. 
		

		The controls available are the RE's state -- on or off -- and the power produced in the RE (and delivered into the powertrain).
		The power delivered into the powertrain is a non-negative piecewise continuous time function.
		The RE's state is controlled by a discrete sequence of switching orders decided and executed at discrete times.
		An important feature of the REEV model is a time interval $\delta > 0$ imposed between two consecutive decisions times. 		
		From the physical viewpoint, this assumption incorporates the fact that frequent switching of the RE is undesirable in order to avoid mechanical wear off and acoustic nuisance for the driver. 
		
		The model considers that the vehicle's traction capability is conditioned to the existence of some electric energy in the battery. 
		Since the vehicle must halt whenever there is no charge available in the battery, the objective of finding the vehicle autonomy is summarized into finding the furthest point away from the vehicle geographic starting point where the battery is depleted for the first time.
		
	\subsection{Hybrid Dynamical System}

		Hybrid systems have some supervision logic that intervenes punctually between two or more continuous functions. 
		The main elements of the class of hybrid dynamical systems considered in this work are a family of continuous dynamics (vector fields) $f$ and continuous state spaces $X$, indexed by a discrete state $q$ valued in a discrete set $Q$.
		Each continuous dynamic system $f_q$ valued in $X_q$ models a physical process controlled by a continuous control function $u$, from which the system is free to switch to another process $f_{q\rq{}}$ using a specified discrete control $w$ and a discrete dynamics $g$.
		
		More precisely, the continuous state variable is denoted $y$ and it is valued in the state space $X = \RRR^d$.
		The discrete variable is $q \in Q = \{0,1,\cdots,d_q \}$, where $d_q$ is the number of possible dynamics that can operate the system.
		(for simplicity of the presentation,  through this paper, we consider that $X_q= X$ for all $q\in Q$).
		Each of these dynamics models a different mode of operation of the system or a different physical process that the system is undergoing.
		Moreover, define a compact set $K \subset X$ as the hybrid system admissible set, i.e., a set inside which the state must remain.
		
		The continuous control is supposed to be a measurable function $u$ valued in a set that depends on the mode that is currently active $U(q)$.
		The discrete control is a sequence of switching decisions
		\begin{equation}
			w =\{ (w_1,s_1), \cdots, (w_i,s_i), (w_{i+1},s_{i+1}),\cdots\},
		\end{equation}•
		where each $s_i \in [0,\infty[$ and $w_i \in W(q) \subset \{0,1, \cdots,d_q\}$.
		The sequence of discrete switching decisions $\{w_i\}_{i>0}$ (designating the new mode of operation) is associated with a sequence of switching times $\{s_i\}_{i>0}$ where each decision $w_i$ is exerted at time $s_i$. 
		The set of available discrete decisions, at time $s$, $W(q(s))$, depends on the discrete state variable and it corresponds to a decision of switching the system to another process $w_i$.
		
		The lag condition between switches is included by demanding that two switch orders must be separated by a time interval of $\delta >0$, i.e., $$s_{i+1}-s_i \geq \delta.$$
		\begin{remark}
			Regarding the vehicle application, the vehicle's energy state is a two-dimensional vector  $y \in X = \RRR^2$, where $y=(y_1,y_2)$ denotes the state of charge of the battery and the fuel available in the range extender module.
			Each of these quantities are the image of the remaining energy in the battery and the RE respectively.
			It is clear here that the state variables have to be constrained  to remain in the compact set $K = [0,1]^2$, where the energies quantities are normalized. 
			$q \in Q = \{0,1\}$ is the RE state, indicating whether the RE is off ($q=0$) or on ($q=1$). 
			The power output is a measurable function $u(\cdot) \in U(q(\cdot))$ where $U(\cdot)$ is the admissible control set, compact subset of $\RRR$ dependent naturally on the RE state.
		\end{remark}

		In this setting, given a discrete state $q$, the continuous control $u$ steers the continuous system
		\begin{equation}
			\dot{y}(\tau) = f(\tau,y(\tau),u(\tau),q_i), \text{ for a.e. } \tau \in [0,t] \label{sysContinuous}
		\end{equation}
	 	where some continuous dynamics $f(\cdot,\cdot,\cdot,q_i)$ is activated. $f$ is a family of vector fields indexed by the discrete variable $q$. When $q=q_i$, the corresponding vector field is active and dictates the evolution of the continuous state. 
		At some isolated times $\{s_i\}_{i>0}$, given by the discrete switching control sequence, the discrete dynamic $g$ is activated
		\begin{equation}
			q_{i-1} = g(w_i,q_i)
		\end{equation}
		and the continuous state follows another vector field $f(\cdot,\cdot,\cdot,q_{i-1})$. 
		In the considered system, the discrete decisions only switch the continuous dynamics and introduce no discontinuities on the trajectory. 
		In more general frameworks, we can also include jumps in the continuous state vector that can be used to model an instant change in the value of the state following a discrete decisions \cite{Branicky,BM}.	
	
		Assume controlled continuous dynamics $f$ and the discrete dynamic $g$ satisfy the following:
		\begin{description}
		\item[(H0)]  The continuous control is a measurable function $u: [0,\infty[ \rightarrow \RRR^m$ such that
		\begin{equation*}
			u(\tau) \in U(q(\tau)) \text{ for a.e. } \tau\in [0,t].
		\end{equation*}

		\item[(H1)] There exists $L_f > 0$ such that, for all $s\geq 0$, $y,y' \in X$, $q \in Q$ and $u \in  U(q)$,
			\begin{equation}	
			\nonumber	\| f(s,y,u,q) - f(s,y',u,q) \| \leq L_f\|y-y'\|, ~~ \|f(s,y,u,q) \| \leq L_f.
			\end{equation}
		\item[(H2)]  For all $q\in Q$, $f(\cdot,\cdot,\cdot,q) : [0,\infty[ \times X \times U \rightarrow X$ is continuous and for all $s \in [0,\infty[, x\in X, u \in U $, $f(s,x,u,\cdot) : Q \to X$ is continuous with respect to the discrete topology.
	
		\item[(H3)]  For all $s\in [0,t]$, $y \in X$ and $q \in Q$, $f(s,y,U,q)$ is a convex subset of $X$.
	
		\item[(H4)] There exists $L_g > 0$ such that, for all $q \in Q$ and $w \in  W(q)$,
			\begin{equation}	
			\nonumber	\| g(w,q) \| \leq L_g 
			\end{equation}
		\item[(H5)]   $g(\cdot,\cdot) $ is continuous with respect to the discrete topology.
		\end{description}	
		
	Assumption (H1) ensures that a trajectory exists and that it is unique. 
	Assumptions (H2)-(H5) are used to prove the Lipschitz continuity of the value function and (H3) is needed in order to observe the compactness of the trajectory space. 
	
	Denote $A$ the space of hybrid controls $a = (u,w)$.
	We precise the class of admissible controls $\AAA \subset A$ in the following definition:
	\begin{definition}
		For a fixed $t\geq 0$ a hybrid control $a = (u,w) \in \AAA$ is said to be admissible if the continuous control verifies (H0) and the discrete control sequence  $w=\{w_i,s_i \}_{i >0}$ has increasing decision times
		\begin{equation}
			s_1 \leq s_2 \leq \cdots \leq s_i \leq s_{i+1} \leq \cdots \leq t,
			\label{discrete1}
		\end{equation}
		admissible decisions
		\begin{equation}
			\forall i > 0 ,~ w_i \in W(q(s_i)) \subset Q,
			\label{discrete2}
		\end{equation}
		and verifies a decision lag
		\begin{equation}
			s_{i+1}-s_i \geq \delta,
			\label{discrete3}
		\end{equation}	
		where $\delta >0$.		
	\end{definition}
	
	An important consequence in the definition of admissible control is the finiteness of the number of switch orders:
	\begin{proposition}
		Fix $s \geq 0$. Let $a \in A$ be an admissible hybrid control. 
		Then, the discrete control sequence has at most $N= \lfloor s / \delta \rfloor$ switch decisions.
	\label{finiteJumps}
	\end{proposition}				
												
	Fix $t>0$. Given a hybrid control $a  \in A$ with $N$ switch orders and given $x\in X$, $q\in Q$, the hybrid dynamical system is
	\begin{eqnarray}
		\dot{y}(\tau) &=& f(\tau,y(\tau),u(\tau),q_i),~~  \tau \in [0,t],~~~~~~~~~~ y(t) = x \label{HybridSystemA}  \\
			q_{i-1} &=& g(w_i,q_i),~~ i=1,\cdots,N,~~~~~~~~~~~~~~~~~\, q_N = q \label{HybridSystemB} 
	\end{eqnarray}

	Denote  the solutions of \eqref{HybridSystemA}-\eqref{HybridSystemB} with final conditions $x,q$ by $y_{x,q;t}$ and $q_{x,q;t}$.
	As pointed out, not all discrete control sequences are admissible.
	Only admissible control sequences engender admissible trajectories.
	Thus, given $t >0$, $x \in X$ and $q \in Q$, the admissible trajectory set $Y^{x,q}_{[0,t]} $  is defined as
	\begin{equation}
		Y^{x,q}_{[0,t]} = \{ y(\cdot)  ~|~ a \in \AAA \text{ and }y_{x,q;t} \text{ solution of \eqref{HybridSystemA}-\eqref{HybridSystemB} }\}
		\label{admTrajAux}
	\end{equation}•
	A consequence of proposition \ref{finiteJumps} and  the above definition is the finiteness of the number of discrete decisions in any admissible trajectory.	
	Observe that the admissible trajectories set does not include the discrete trajectory.
	
	The hybrid control admissibility condition formulated as in conditions \eqref{discrete1}-\eqref{discrete3} is not well adapted to a dynamic programming principle formulation, needed later on.
	In order to include the admissibility condition in the optimal control problem in a more suitable form, we introduce a new state variable $\pi$.
	Recall that the decision lag conditions implies that new switch orders are not available up to a time $\delta$ since the last switch.
	The new variable is constructed such that at a given time $\tau \in [0,t]$, the value of $\pi(\tau)$ measures the time since the last switch.
	The idea is to impose constraints on this new state variable and treat them more easily in the dynamic programming principle.
	Thus, if $\pi(\tau)<\delta$ all switch decisions are blocked and if, conversely, $\pi(\tau) \geq \delta$ the system is free to switch.
	For that reason, this variable can be seen as a switch lock.

	Now given $t >0$, $\tau \in [0,t]$ and a discrete control $w=\{w_i,s_i \}_{i > 0}$, the switch lock dynamics is defined by
	\begin{equation}	
		\pi^w(\tau) = \pi(\tau) = \left\{ \begin{array}{lll}
							\delta + \tau & \text{if} & \tau < s_1 \\
							\inf_{s_i \leq s} \tau-s_i & \text{if} & \tau \geq s_1 \\
						 \end{array} \right.
	\label{pDynamic}
	\end{equation}
	Indeed, once the discrete control is given, the trajectory $\pi(\cdot)$ can be determined.
	Proceeding with the idea of adapting the admissibility condition in order to manipulate it in a dynamic programming principle, we wish to consider $\pi(t)=p$, with $p \in P := (0,t]$, the final value of the switch lock variable trajectory and impose the lag condition under the form $\pi(s_i^-) \geq \delta$ for all $s_i$, where $s_i^-$ denotes the limit to the left at the switching times $s_i$ (notice that $\pi(s_i^+)=0$ by construction).
	Then, since these conditions suffice to define an admissible discrete control set, while optimizing with respect to admissible functions, one needs only look within the set of hybrid controls that engenders a trajectory $\pi(\cdot)$ with the appropriate structure.
	In other words, given $t >0$, $x \in X$, $q \in Q$ and $p \in P$, define a admissible trajectory set $ S_{[0,t]}^{x,q,p} $ as
	\begin{eqnarray}	
		\nonumber S_{[0,t]}^{x,q,p} &=& \{ y(\cdot) ~|~ a = (u,\{w_i,s_i \}_{i =1}^{N}) \in A, ~y_{x,q,p;t} \text{ solution of \eqref{HybridSystemA}-\eqref{HybridSystemB}}, ~~~~~~~~~~ \\
					&&	~~~~~~~~\pi(\cdot) \text{ solution of \eqref{pDynamic}}, ~\pi(t) = p, ~\pi(s_i^-) \geq \delta, ~ i=1,\cdots,N \}.
		\label{admTrajSet}
	\end{eqnarray}
	
	The next lemma states a relation between sets $Y$ and $S$:
	
	\begin{lemma}
		Following the above definitions, sets \eqref{admTrajAux} and \eqref{admTrajSet} satisfy
		\begin{equation*}
			Y^{x,q}_{[0,t]} = \bigcup_{p \in P} S_{[0,t]}^{x,q,p}
		\end{equation*}•
	\end{lemma}
	\begin{proof}
		The equivalence between $Y^{x,q}_{[0,t]}$ and $\bigcup_{p \in P} S_{[0,t]}^{x,q,p}$ is obtained by construction.
	\end{proof}

	In the following of the paper, whenever we wish to call attention to the fact that the final conditions of \eqref{HybridSystemA}, \eqref{HybridSystemB} and \eqref{pDynamic} are fixed, we denote their solutions respectively by $y_{x,q,p;t},q_{x,q,p;t},\pi_{x,q,p;t}$.

	\subsection{Reachability of Hybrid Dynamical Systems and Optimal Control Problem}
	
	Let $X_0 \subset X$ be the set of allowed initial states, i.e. the set of states from which the system \eqref{HybridSystemA}-\eqref{HybridSystemB} is allowed to start.
	Define the reachable set as the set of all points attainable by $y$ after a time $s$ starting within the set of allowed initial states $X_0$ to be
	\begin{eqnarray}
	\nonumber	R_{X^0}(s) 	&=& \{ x ~|~ \exists q \in Q,~   y_{x,q;s} \in Y^{x,q}_{[0,s]}, ~y_{x,q;s}(0) \in X^0, \text{ and } y_{x,q;s}(\theta) \in K,~ \forall \theta \in [0,s] \}  \\
	\nonumber		&=& \{ x ~|~ \exists (q,p) \in Q\times P, ~   y_{x,q,p;s} \in S^{x,q,p}_{[0,s]}, ~y_{x,q,p;s}(0) \in X^0 \\
				& & \hspace{5cm} \text{ and } y_{x,q,p;s}(\theta) \in K,~ \forall \theta \in [0,s] \}.
	\label{reachForward}
	\end{eqnarray}

	In other words, the reachable set $R_{X^0}(s)$ contains the values of $y_{x,q;s}(s)$, regardless of the final discrete state, for all admissible trajectories -- i.e., trajectories obtained through an admissible hybrid control -- starting within the set of possible initial states $X_0$, that never leave set $K$. 
	
	Observe that \eqref{reachForward} defines the reachable set $R_{X^0}(s) $ in terms of both admissible trajectory sets $Y$ and $S$.

	\begin{remark}
		In particular, the information contained in \eqref{reachForward} allows one to determine the first time where the reachable set is empty.
	More precisely, given $X_0 \subset X$, define $s^* \geq 0$ to be
	\begin{equation}
		s^* = \inf \{ s ~|~ R_{X^0}(s) \subset \emptyset \}.
	\label{hybridAutonomy}
	\end{equation}•
	The time \eqref{hybridAutonomy} is the autonomy of the hybrid system \eqref{HybridSystemA}-\eqref{HybridSystemB}.
	Indeed, one can readily see that if no more admissible energy states are attainable after $s^*$, any admissible trajectory must come to a stop beyond this time.
	Therefore, $s^*$ is the longest time during which the state remains inside $K$.
\end{remark}

	The following proposition ensures that the space of admissible trajectories is a compact set.

	\begin{proposition}
		Given $T' >0$, the admissible trajectory set $Y^{x,q}_{[0,T']}$ is a compact set in $C([0,T\rq{}])$ endowed with the topology $W^{1,1}$.
		\label{lemmaINF}
	\end{proposition}

	\begin{proof}
		Fix $q \in Q$ and $0 \leq s < t \leq T\rq{}$.
		Consider a bounded admissible continuous control sequence $u_n \in L^1([s,t])$. Since $u_n$ is bounded, there exists a subsequence  $u_{n_j}$ such that $u_{n_j} \rightharpoonup u$ in $ L^1([s,t])$.
		Invoking $(H1)-(H2)$, we have $y^{u_n} = y_n \rightharpoonup y$ in $W^{1,1}([s,t])$. 
		Since $W^{1,1}([s,t])$ is compactly embedded in $C^0([0,t])$, we get the strong convergence of the solution $y_n \rightarrow y$ in $C^0([s,t])$. 
		Hypothesis $(H3)$ guarantees that the limit function $y$ is a solution of \eqref{HybridSystemA}.
		Because all controls $u_n$ and the limit control $u$ are admissible, $y$ is an admissible solution. 
		
		So far, the proof shows that the limit trajectory is admissible when $q$ is hold constant.
		Consider a sequence of admissible discrete control sequences $(w)_n$ where the number of switching orders, $0\leq k_n \leq \left\lfloor T / \delta \right\rfloor$ may depend on $n$. 
		Since each term of this sequence has a (first) discrete component and is bounded on the (second) continuous component, then, as $n \rightarrow \infty$ there exists a subsequence $(w)_{n_l}$ and $\Lambda>0$ such that $q_{n_j}=q$ for all $l > \Lambda$.
		This implies $k_n \rightarrow k$.
		As the number of switches is constant from the $\Lambda^\text{th}$ term and $s,t$ are arbitrary, one can obtain, using the limit discrete control sequence $w$, the time intervals $[s_{i-1},s_i]$ over which $q_i$ is constant and the argument in the first paragraph of the proof.
		Because the trajectory is continuous and admissible on all time intervals $[s_{i-1},s_i], i=1,\cdots,K$, it is admissible on $[0,T]$.

		Moreover, observe that $\forall s \in [0,t]$, $y_n(s) \in K$. Since for all $s\in [0,t]$, $y_n(s) \rightarrow y(s)$, by the compactness of $K$ we get that $\forall s\in [0,t],$ $y(s) \in K$, which completes the proof.
	\end{proof}
	
	\begin{remark}
		The arguments presented in the above proof can be slightly modified to show that the admissible trajectory set with fixed final $p$, $S^{x,q,p}_{[0,T']}$ is compact.
		Also in a similar way, the proof can be adapted to show that the reachable $R_{X^0}$ is closed.
		Indeed, by the compactness of set $X_0$, a sequence of initial conditions $(y_0)_n \in X_0$, associated with admissible trajectories $y_n \in Y^{x,q}_{[0,T']}$, converges to $y_0 \in X_0$ which is also the initial condition for the limiting trajectory $y_n \to y$.
	\end{remark}

	In order to characterize the reachable set $R_{X_0}$ this paper follows the classic level-set approach \cite{OsherSethian88}.
	The idea is to describe \eqref{reachForward} as the negative region of a function $v$. 
	It is well known that the function $v$ can be defined as the value function of some optimal control problem. 
	In the case of system \eqref{HybridSystemA}-\eqref{HybridSystemB}, $v$ happens to be the value function of a hybrid optimal control problem.
	
	Consider a Lipschitz continuous function $\tilde \phi : X \rightarrow \RRR$ such that
	\begin{equation*}
		\tilde \phi(x) \leq 0 \Leftrightarrow x \in X^0.
	\end{equation*}
	Such a function always exists -- for instance, the signed distance function $d_{X_0}$ from the set $X^0$.
	For $ L_K >0$, one can construct a bounded function $ \phi : X \rightarrow \RRR$ as 
	\begin{equation}
		\phi(x) = \max(\min(\tilde \phi(x),  L_K),- L_K).
		\label{distance}
	\end{equation}

	For a given point $s\geq 0$ and hybrid state vector $(x,q,p) \in X \times Q \times P$, define the value function to be
	\begin{equation}
		v_0(x,q,p,s) = \inf_{ S_{[0,s]}^{x,q,p}} \left\{ \phi(y_{x,q,p;s}(0)) ~|~ y_{x,q,p;s}(\theta) \in K,~ \forall \theta \in [0,s] \right\}
		\label{valueConst}
	\end{equation}
	Observe that \eqref{reachForward} works as a level-set to the negative part of \eqref{valueConst}.
	Indeed, since \eqref{valueConst} contains only admissible trajectories that remain in $K$, by \eqref{distance} implies that $v_0(x,q,p,s)$ is negative if and only if $y_{x,q,p;s}(0)$ is inside $X_0$, which in turn implies that $x \in R_{X^0}(s)$.
	
	Remark however that when defining the value function with \eqref{valueConst}, one includes state constraints, with the condition that $y_{x,q,p;s}(\theta) \in K$ for all times.
	When $K \neq \RRR^d$, one cannot expect $v$ to be continuous and the HJ equation associated with \eqref{valueConst} may have several solutions.
	In order to bypass such regularity issues, this paper follows the idea of \cite{BFZ,BFZ2}.
	Define a Lipschitz continuous function $\tilde \varphi: X \rightarrow \RRR$ to be
	\begin{equation*}
		\tilde \varphi(x) \leq 0 \Leftrightarrow x \in K,
	\end{equation*}
	and
	\begin{equation}
		\varphi(x) = \max(\min(\tilde \varphi(x),  L_K), -  L_K).
		\label{penal}
	\end{equation}
	Then, for a given $s\geq 0$ and $(x,q,p) \in X \times Q \times P$, define a total penalization function to be
	\begin{equation*}
		J(x,q,p,s ; y) = \left( \phi(y_{x,q,p;s}(0)) \bigvee \max_{\theta \in [0,s]} \varphi(y_{x,q,p;s}(\theta)) \right)
	\end{equation*}•
	and then, the optimal value :
	\begin{equation}
		v(x,q,p,s) = \inf_{y\in S_{[0,s]}^{x,q,p}} J(x,q,p,s ; y).
		\label{valuePenal}
	\end{equation}
	
	Observe that \eqref{valueConst} and \eqref{valuePenal} are bounded thanks to the constructions \eqref{distance} and \eqref{penal} respectively.
	The idea in place is that one needs only to look at the sign of $v_0$ or $v$ to obtain information about the reachable set.
	Therefore, the bound $L_K$ removes the necessity of dealing with an unbounded value function besides providing a convenient value for numerical computations.
	In order to ensure that constructions \eqref{distance} and \eqref{penal} do not interfere with the original problem\rq{}s formulation (in the sense that using $\phi,\varphi$ or $\tilde \phi, \tilde \varphi$ should yield the same results), given $s,X_0$ and $x$, assuming one does use signed distance functions to sets $X_0,K$, it suffices to take $L_K >  \sup_{x\rq{} \in X_0} x\rq{} e^{L_f  |x - x\rq{}| s}$, where $L_f$ is the Lipschitz constant of $f$.

	\section{Main Results}


	The next proposition certifies that \eqref{reachForward} is indeed a level-set of \eqref{valueConst} and \eqref{valuePenal}.	
	\begin{proposition}
	Assume (H1)-(H3). 
	Define Lipschitz continuous functions $\phi$ and $\varphi$ by \eqref{distance} and \eqref{penal} respectively. 
	Define value functions $v_0$ and $v$ by \eqref{valueConst} and \eqref{valuePenal} respectively. 
	Then, for $s \geq 0$, the reachable set is given by
		\begin{equation}
			R_{X^0}(s) = \{ x ~|~ \exists (q,p) \in Q\times P, ~v_0(x,q,p,s) \leq 0 \} = \{ x ~|~ \exists (q,p) \in Q\times P,  ~v(x,q,p,s) \leq 0 \}
		\end{equation}
	\label{levelSet}
	\end{proposition}
	\begin{proof}
		The proof begins by showing that $v_0(x,q,p,s) \leq 0 \Rightarrow v(x,q,p,s) \leq 0$. 
		Assume $v_0(x,q,p,s) \leq 0$. 
		Then, using lemma \ref{lemmaINF}, there exists an admissible trajectory such that
		\begin{equation*}
		\phi(y_{x,q,p;s}(0)) \leq 0,~ y_{x,q,p;s}(\theta) \in K,~ \forall \theta \in [0,s]. 
		\end{equation*}
		Thus, $\max_{\theta \in [0,s]} \varphi(y_{x,q,p;s}(\theta)) \leq 0$ and
		\begin{equation*}
			v(x,q,p,s) \leq \max(\phi(y_{x,q,p;s}(0)), \max_{\theta \in [0,s]} \varphi(y_{x,q,p;s}(\theta)) ) \leq 0
		\end{equation*}
		
		Now, show that $v(x,q,p,s) \leq 0 \Rightarrow v_0(x,q,p,s) \leq 0$.
		Assume $v(x,q,p,s) \leq 0$. 
		Then, by lemma \ref{lemmaINF} there exists a trajectory that verifies
		\begin{equation*}
			\max(\phi(y_{x,q,p;s}(0)), \max_{\theta \in [0,s]} \varphi(y_{x,q,p;s}(\theta)) ) \leq 0.
		\end{equation*}
		By the definition of $\varphi$, $\forall \theta \in [0,s]$,
		$$\max (\varphi(y_{x,q,p;s}(\theta))) \leq 0 \Rightarrow y_{x,q,p;s}(\theta) \in K, $$
		which implies $v_0(x,q,p,s) \leq 0$.
		Therefore, $u$ and $v$ have the same negative regions.
		
		Now, assume $y_{x,q,p;s}(s) \in R_{X^0}(s) $. 
		Then, by definition, there exists  $(q,p) \in Q\times P$ and an admissible trajectory such that $y_{x,q,p;s}(\theta) \in K$ for all time and $y_{x,q,p;s}(0) \in X_0$.
		This implies that $\max_{ \theta \in [0,s]}( \varphi(y_{x,q,p;s}(\theta))) \leq 0$ and $\phi(y_{x,q,p;s}(0)) \leq 0$.
		It follows that $v(x,q,p,s) \leq  J(x,q,p,s;y) \leq 0$.
				
		Conversely, assume $v(x,q,p,s) \leq 0$.
		For any optimal trajectory $\hat y$ (which is admissible thanks to proposition  \ref{lemmaINF}) $v(x,q,p,s) =  J(x,q,p,s;\hat y) \leq 0$.
		Since the maximum of the two quantities is non positive only if they are both non positive one can draw the desired conclusion.			
	\end{proof}
	
	Proposition \eqref{levelSet} sets the equivalence between \eqref{reachForward} and the negative regions of \eqref{valueConst} and \eqref{valuePenal}.
	In particular, it states that it suffices to computes $v$ or $v_0$ in order to obtain information about $R_{X_0}$.
	In this sense, this paper focuses on \eqref{valuePenal}, which is associated with an optimal control problem with no state constraints.


	In the sequel, it is shown that \eqref{valuePenal} is the unique (viscosity) solution of a quasi-variational inequalities\rq{} system. 
	The first step is to state a dynamic programming principle for \eqref{valuePenal}.
	
	First, we present some preliminary notation.
	Given $t >0$, set $T = (0,t]$, $\Omega = X \times Q \times P \times T$ and denote its closure by $\overline \Omega$. 
	For a fixed $p_0 \in P$, define $\Omega |_{p_0} = X \times Q \times \{p_0\} \times T$ and denote $\overline \Omega |_{p_0}$ the closure of $\Omega |_{p_0}$.
	Define
	\begin{eqnarray}
		\VVV(\overline\Omega) &:=& \{ v ~|~ v: \overline\Omega \to \RRR, ~v \text{ bounded } \}, \\
		\VVV(\overline\Omega|_{p_0}) &:=& \{ v ~|~ \text{for } p_0 \in P, ~v: \overline\Omega|_{p_0} \to \RRR, ~v \text{ bounded } \}
	\end{eqnarray}
	For $v \in \VVV(\overline \Omega)$, denote its upper and lower envelope at point $(x,q,p,s) \in \overline \Omega$ respectively as $v^*$ and $v_*$:
	\begin{eqnarray}
		v^*(x,q,p,s) &=& \limsup\limits_{\substack{x_n \to x \\ q_n \to q \\ p_n \to p \\ s_n \to s }} v(x_n,q_n,p_n,s_n)  	\label{limSupDef}\\
		v_*(x,q,p,s) &=& \liminf\limits_{\substack{x_n \to x \\ q_n \to q \\ p_n \to p \\ s_n \to s }} v(x_n,q_n,p_n,s_n) 	\label{limInfDef}
	\end{eqnarray}
	In the case where $p_0\in P$ is fixed and $v \in \VVV(\overline \Omega |_{p_0} )$, the upper and lower envelopes of $v$ are also given by \eqref{limSupDef}, \eqref{limInfDef} with $p_n = p_0$ for all $n$.
	
 	Now, fix $p=0$ and define the non-local switch operators $M,M^+,M^-: \VVV(\overline \Omega |_0) \rightarrow \VVV(\overline \Omega|_0)$ to be
	\begin{eqnarray*}
		(Mv)(x,q,0,s) &=& \inf_{\substack{w\in W(q) \\ p' \geq \delta}} v(x,g(w,q),p',s) \\
		(M^+v)(x,q,0,s) &=& \inf_{\substack{w\in W(q) \\ p' \geq \delta}}  v^*(x,g(w,q),p',s) \\
		(M^-v)(x,q,0,s) &=& \inf_{\substack{w\in W(q) \\ p' \geq \delta}}  v_*(x,g(w,q),p',s)
	\end{eqnarray*}
		
	The action of these operators on the value function represents a switch that respects the lag constraint.
	They operate whenever a switch is activated, which is equivalent to the condition $p=0$.
	Therefore, they are defined only for a fixed $p=0$.
	Let us recall here some classical properties of operators $M,M^+$ and $M^-$ (adapted from \cite{Zhang}):
	
	\begin{lemma}
	Let $v\in B\VVV(\closure{\, \Omega})$. 
	Then $M^+ v^* \in BUSC(\closure{\, \Omega}) $ and $M^- v_* \in BLSC(\closure{\, \Omega})$.	
	Moreover $(Mv)^* \leq M^+v^*$ and $(Mv)_* \geq M^-v_*$. 
	\label{Moperator}
	\end{lemma}
	
	\begin{proof}
		Fix $q \in Q$, $ p=0$ and $\epsilon >0$. 
		Let $w^* \in W(x,q)$ and $p^*>0$ be such that for all $x \in X$ and $s \in T$, $(M^+v^*)(x,q,p,s) \geq v^*(x,g(w^*,q),p^*,s) - \epsilon$.
		Consider sequences $x_n \to x$ and $s_n \to s$.
		Then
		\begin{eqnarray*}
			M^+v^*(x,q,0,s) &\geq& v^*(x,g(w^*,q),p^*,s) - \epsilon\\
				&\geq& \limsup\limits_{\substack{x_n \to x \\ s_n \to s }} v^*(x_n,g(w^*,q),p^*,s_n)  - \epsilon \\
				&\geq& \limsup\limits_{\substack{x_n \to x \\ s_n \to s }} \inf_{w \in W(x_n,q) , p\rq{} \geq \delta} v^*(x_n,g(w,q),p\rq{},s_n)  - \epsilon \\
				&=& \limsup\limits_{\substack{x_n \to x \\ s_n \to s }} (M^+v^*)(x_n,q,0,s_n)  - \epsilon.
		\end{eqnarray*}•
		Notice that $q$ and $p$ are held constant throughout the inequalities and thus, the limsup of the jump operator considering only sequences $x_n\to x$ and $s_n\to s$ corresponds to its envelope at the limit point.
		Then, by the arbitrariness of $\epsilon$, this proves the upper semi-continuity of $M^+ v^*$.
		The lower semi-continuity of $M^- v_*$ can be obtained in a similar fashion.
		
		Now, observe that $Mv \leq M^+v^*$.
		Taking the upper envelope of each side, one obtains:
		\begin{equation*}	
			(Mv)^* \leq (M^+v^*)^* = M^+v^*.
		\end{equation*}•
		By the same kind of reasoning  $Mv \geq M^-v_*$ and
		\begin{equation*}	
			(Mv)_* \geq (M^-v_*)_* = M^-v_*.
		\end{equation*}•	
	\end{proof}

	The next proposition is the dynamic programming principle verified by \eqref{valuePenal}: 
	
	\begin{proposition}
	The value function \eqref{valuePenal} satisfies the following dynamic programming principle:
	\begin{enumerate}
		\item[(i)] For s=0, 
		\begin{equation}
			v(x,q,p,0) = \max(\phi(x) , \varphi(x)),~~ \forall (x,q,p) \in X \times Q \times P,
		\label{initDPP}
		\end{equation}	
		\item[(ii)] For $p=0$, 
		\begin{equation}
				v(x,q,0,s) =(Mv)(x,q,0,s), ~~ (x,q,s) \in X \times Q \times T,
		\label{switchPDD}
		\end{equation}
		\item[(iii)] For $(x,q,p,s) \in \Omega$, define the non-intervention zone as $\Sigma = (0,p \wedge s)$.
		Then, for $h \in \Sigma$,
		\begin{equation}
			v(x,q,p,s) = \inf_{S_{[s-h,s]}^{x,q,p}} \left\{ v(y_{x,q,p;h}(s-h),q,p-h,s-h) \bigvee \max_{\theta \in [s-h,s]} \varphi(y_{x,q,p;h}(\theta))  \right\}
		\label{eqPDD}
		\end{equation}

	\end{enumerate}	
	\label{LemmaPDD}
	\end{proposition}

	\begin{proof}
	The dynamic programming principle is composed of three parts. 
	
	(i):	
	Equality \eqref{initDPP} is obtained directly by definition \eqref{valuePenal}.
	\\
	
	(ii):	
	\rq\rq{} $\leq$ \rq\rq{}.
	Let $ (x,q,s) \in X \times Q \times T$ and $p=0$.
	Consider a hybrid control $a = (u(\cdot),\{ w_i,s_i \}_{i=1}^N)$ and an associated trajectory $y^a$.
	Construct a control $\overline a = (\overline u(\cdot),\{ \overline w_i,\overline s_i \}_{i=1}^{N-1})$ with associated trajectory $y^{\overline a}$, where $\overline u =u$, $\overline w_i = w_{i}$ and  $\overline s_i = s_{i}$ for $i=1,\cdots,N-1$ and $s_N = s$, $w_N = w\rq{}$.
	Then, one obtains,
	\begin{eqnarray*}
		v(x,q,0,s) 	&\leq& J(x,q,0,s;y^a) \\
			 	&=& J(x,g(w\rq{},q),p\rq{} ,s; y^{\overline a}),
	\end{eqnarray*}	
	where the controller must respect the condition $p\rq{}\geq \delta$ for it to be admissible.
	Since $\overline a$ is arbitrary, one can choose it such that
	\begin{eqnarray*}
		v(x,q,0,s) 	&\leq& \inf_{y^{\overline a} \in S_{[0,s]}^{(x,g(w\rq{},q),p\rq{})}} J(x,g(w\rq{},q),p\rq{} ,s; y^{\overline a}) +\epsilon_1 \\
				&=&	v(x,g(w\rq{},q),p\rq{},s) + \epsilon_1,
	\end{eqnarray*}
	where $\epsilon_1 >0$.
	Now, choose the last switch $w\rq{}$ and $p\rq{}$ such that
	\begin{eqnarray*}
		v(x,q,0,s) 	&\leq&   \inf_{\substack{w\rq{}\in W(q) \\ p' \geq \delta}} v(x,g(w\rq{},q),p',s) +\epsilon_1 \\
				&=& (Mv)(x,q,p,s) + \epsilon_1.
	\end{eqnarray*}
	
	\rq\rq{} $\geq$ \rq\rq{}.
	For $ (x,q,p) \in X \times Q \times P$ and $p=0$, there always exists an admissible control that $a_\epsilon$, such that there exists $\epsilon_2 >0$ and,
	\begin{equation*}
	v(x,q,0,s) +  \epsilon_2 \geq  J(x,q,0,s;y^{a_\epsilon})
	\end{equation*}
	Using the same hybrid control constructions as in the \lq\lq{}$\leq$\rq\rq{} case, one obtains 
	\begin{eqnarray*}
		 J(x,q,0,s;y^{a_\epsilon}) &=& J(x,g(w\rq{},q),p\rq{},s ; y^{\bar a_\epsilon})\\
				&\geq& v(x,g(w\rq{},q),p\rq{},s) \\
				&\geq& \inf_{\substack{w\rq{}\in W(q) \\ p' \geq \delta}} v(x,g(w\rq{},q),p\rq{},s) \\
				&=& (Mv)(x,q,0,s).
	\end{eqnarray*}	
	Relation \eqref{switchPDD} is obtained by the arbitrariness of both $\epsilon_1, \epsilon_2$.
	\\
	
	(iii):
	\rq\rq{} $\leq$ \rq\rq{}.
	For $(x,q,p,s) \in \Omega$ and $0<h\leq p \wedge s$, \eqref{valuePenal} yields
	\begin{equation}
		v(x,q,p,s) \leq \max  \left(  \left(  \phi(y_{x,q,p;s}(0)) \bigvee \max_{\theta \in [0,s-h]} \varphi(y_{x,q,p;s}(\theta)) \right)  , \max_{\theta \in [s-h,s]} \varphi(y_{x,q,p;s}(\theta))  \right),
		\label{ineqValueDecomp}
	\end{equation}
	for any $y \in  S^{x,q,p}_{[0,s]}$.
	By the choice of $h$, there is no switching between times $s-h$ and $s$.
	Write the admissible control $a=(u,w)$ as $a_0 = (u_0,w)$ and $a_1 = (u_1,w)$ with
	\begin{eqnarray*}
		u_0(s) = u(s), ~~ s\in[0,s-h], \\
		u_1(s) = u(s),~~ s\in(s-h,s].
	\end{eqnarray*}
	Since $a$ is admissible, both controls $a_0,a_1$ are also admissible.
	Denote the trajectory associated with controls $a,a_0,a_1$ respectively by $y^a, y^0, y^1$.
	Then, $y^a \in S^{x,q,p}_{[0,s]}$ and by continuity of the trajectory we achieve the following decomposition:
	\begin{eqnarray*}
		y^1 \in S^{x,q,p}_{[s-h,s]}, ~~~~ y^0 \in S^{y^1(s-h),q,p-h}_{[0,s-h]}.
	\end{eqnarray*}
	The above decomposition together with inequality \eqref{ineqValueDecomp} yields
	\begin{equation*}	
		v(x,q,p,s) \leq \max  \left(  \left(  \phi(y^0(0)) \bigvee \max_{\theta \in [0,s-h]} \varphi(y^0(\theta)) \right)  , \max_{\theta \in [s-h,s]} \varphi(y^1(\theta))  \right),
	\end{equation*}

	And one concludes after minimizing with respect to the trajectories associated with $a_0$ and $a_1$.
	
	The \rq\rq{} $\geq$ \rq\rq{} part uses a particular $\epsilon$-optimal controller and the same decomposition, allowing to conclude by the arbitrariness of $\epsilon$.	
	This is possible because there is no switching between $s-h$ and $s$.	
	
	\end{proof}
			
	A direct consequence of proposition \ref{LemmaPDD} is the Lipschitz continuity of the value function, stated in the next proposition:
	\begin{proposition}
		Assume (H1)-(H2). 
		Define Lipschitz continuous functions $\phi$ and $\varphi$ by \eqref{distance} and \eqref{penal}, with Lipschitz constants $L_\phi$ and $L_\varphi$ respectively. 
		Then, for $p>0$, \eqref{valuePenal} is Lipschitz continuous.
	\end{proposition}
	\begin{proof}
		Fix $s,s\rq{}>0$, $x,x\rq{} \in X$, $q \in Q$, $p > 0$. 
		Then, using $\max(A,B) - \max(C,D) \leq \max(A-B,C-D) $, one obtains,
		\begin{eqnarray*}
		| v(x,q,p,s) - v(x\rq{},q,p,s) | 	&\leq&  \max \left( \left| \phi(y_{x,q,p;s}(0)) - \phi(y_{x',q,p;s}(0)) \right|, \phantom{ \max_{\theta}}  \right. \\
						&& \hspace{3cm} \left. \max_{\theta \in [0,s]}  \left|\varphi(y_{x,q,p;s}(\theta)) - \varphi(y_{x',q,p;s}(\theta) )\right|   \right)  \\ 
						&\leq&  \max \left( L_\phi \left| y_{x,q,p;s}(0) - y_{x',q,p;s}(0) \right| ,\phantom{ \max_{\theta}}  \right. \\
						&&  \hspace{3.2cm} \left. L_\varphi \max_{\theta \in [0,s]} \left( \left| y_{x,q,p;s}(\theta) - y_{x',q,p;s}(\theta)  \right| \right) \right)  \\
						& \leq& L_v|x-x'|,	
		\end{eqnarray*}
		where $L_v = \max(L_\phi,L_\varphi) e^{L_f s }$.
		
		Now, take $h>0$ and observe that $v(x,q,p,s) \geq \varphi(x)$.
		Then,
		\begin{eqnarray*}
		| v(x,q,p+h,s+h) - v(x,q,p,s) | 	&\leq& \max \left( \left| v(y_{x,q,p+h;s+h}(s),q,p,s) - v(x,q,p,s)  \phantom{ \max_{\theta}} \hspace{-0.6cm} \right| ,\right. \\
						&& \hspace{2.1cm} \left. \left| \max_{\theta \in [s,s+h]} \varphi(y_{x,q,p;s+h}(\theta))-\varphi(x) \right|  \right) \\
						&\leq&  \max \left( L_v \left| y_{x,q,,p+h;s+h}(s)-y_{x,q,p;s}(s) \phantom{ \max_{\theta}} \hspace{-0.6cm} \right| \right. , \\
						&&  \hspace{1.1cm} L_\varphi \left.  \max_{\theta \in [s,s+h]} \left|y_{s,q,p+h;s+h}(\theta)-y_{x,q,p;s}(s) \right| \right)  	\\
						&\leq& L_f \max(L_v,L_\varphi)h.
		\end{eqnarray*}
	\end{proof}

	In order to proceed to the HJB equations, define the Hamiltionian to be
	\begin{equation}
		H(s,x,q,z) =  \sup_{u\in U(q)}  f(s,x,u,q) \cdot z
	\label{Hamilt}
	\end{equation}

	Before stating the next result, we recall the notion of viscosity solution \cite{Crandall92user’sguide} used throughout this paper.	
	
	\begin{definition}
	A function $u_1$ (resp. $u_2$) upper semi-continuous (u.s.c.) (resp. lower semi-continuous (l.s.c) is a viscosity subsolution (resp. supersolution) if there exists a continuously differentiable function $\psi$ such that $u_1 - \psi$ has a local maximum (resp. $u_2 - \psi$ has a local minimum) at $(x,q,p,s) \in \overline \Omega$ and
			\begin{eqnarray}
				\partial_s \psi + \partial_p \psi+ H(s,x,q,\nabla_x \psi)  \bigwedge u_1-\varphi(x)  \leq 0 &\text{ if }& (x,q,p,s) \in \Omega \label{sub} \\
				u_1(x,q,p,s) \leq (M^+u_1)(x,q,p,s) &\text{ if }& p=0, \\
				u_1(x,q,p,s) \leq \max(\phi(x),\varphi(x)) &\text{ if }& s=0,
			\end{eqnarray}
		(with the inequalities signs inversed and $M^-$ instead of $M^+$ for $u_2$).
		A bounded function $u$ is a (viscosity) solution of \eqref{HJ1}-\eqref{HJ3} if $u^*$ is a subsolution and $u_*$ is a supersolution.
	\end{definition}

	The next statement shows that the value function defined in \eqref{valuePenal} is a solution of a quasi-variational system.
	
	\begin{theorem}
	Assume (H1)-(H5). 
	Let the Lipschitz functions $\phi$ and $\varphi$ be defined by \eqref{distance} and \eqref{penal} respectively. 
	Then, the Lipschitz, bounded value function $v$ defined in \eqref{valuePenal} is a viscosity solution of the quasi-variational inequality
	
	\begin{eqnarray}
		\partial_s v + \partial_p v + H(s,x,q,\nabla_x v) \bigwedge v-\varphi(x)  = 0 	&,& \forall (x,q,p,s) \in \Omega, \label{HJ1}\\
		v(x,q,0,s) = (Mv)(x,q,0,s)							&,& \forall (x,q,s) \in X \times Q \times [0,\infty[ \label{HJ2} \\
		v(x,q,p,0)  = \max(\phi(x),\varphi(x))						&,& \forall (x,q,p) \in X \times Q \times P.
		\label{HJ3} 
	\end{eqnarray}
	\label{HJprop}
	\end{theorem}
	\begin{proof}
		By definition, $v$ satisfies the initial condition \eqref{HJ3}. 
		The boundary condition \eqref{HJ2} is deducted from proposition \ref{LemmaPDD}.
		Now, we proceed to show that $(i)$ $v$ is a supersolution and $(ii)$ a subsolution of \eqref{HJ1}:

		First, let us prove the supersolution property $(i)$. 
		To satisfy $\min(A,B) \geq 0$ one needs to show $A\geq 0$ and $B\geq 0$. 
		Since $v-\varphi(x)\geq 0$, it is immediate that $B\geq 0$. 
		Now, consider $0 < h \leq p\wedge s $.
		Let $\psi$ be a continuously differentiable function such that $v-\psi$ attains a minimum at $(x,q,p,s)$.
		Then,using proposition \eqref{LemmaPDD}$(iii)$ and selecting an $\epsilon$-optimal controller, dependent on $h$, with associated trajectory $y^\epsilon_{x,q,p;s}$, it follows that
		\begin{eqnarray*}
			 \psi(x,q,p,s) = v(x,q,p,s) &\geq& \inf_{S_{[0,s]}^{x,q,p}} v(y_{x,q,p;s}(s-h),q,p-h,s-h) \\
		 			 &\geq& v(y^\epsilon_{x,q,p;s}(s-h),q,p-h,s-h) - h\epsilon \\
		 			 &=& \psi(y^\epsilon_{x,q,p;s}(s-h),q,p-h,s-h) - h\epsilon 
		\end{eqnarray*}
		and then,
		$$ \psi(x,q,p,s) -\psi(y^\epsilon_{x,q,p;s}(s-h),q,p-h,s-h) \geq - h \epsilon. $$		
		Since the control domain is bounded and using the continuity of $f,p$ and $\psi$ we divide by $h$ and take the limit $h \rightarrow 0$ to obtain
		$$ \partial_s \psi + \partial_p \psi + H(s,x,q,\nabla_x \psi) \geq -\epsilon $$
		and conclude that $A\geq 0$ by the arbitrariness of $\epsilon$.
		
		For $(ii)$, observe that for $\min(A,B) \leq 0$ it suffices to show that $A\leq 0$ or $B \leq 0$. 
		If $v(x,q,p,s)=\varphi(x)$, it implies $B \leq 0$. 
		On the contrary, if $v(x,q,p,s) > \varphi(x)$, then there exists a $\Sigma \ni h \geq 0 $ small enough so that
		$$ v(y^u_{x,q,p;s}(s-h),q,p-h,s-h) > \max_{\theta \in [s-h,h]} \varphi(y^u_{x,q,p;s}(\theta)) $$
		strictly, using the Lipschitz continuity of $f,p$ and the compactness of $U$ (which ensures the trajectories will remain near each other). 
		Thus, proposition \ref{LemmaPDD}$(iii)$ yields
		$$v(x,q,p,s) = \inf_{y^u \in S^{x,q,p}_{[s-h,s]}}v(y^u_{x,q,p;s}(s-h),q,p-h,s-h).$$
		Fix an arbitrary $u\in U$ and consider a constant control $u(s)=u$ for $ 0< s <h$. 
		Let $\psi$ be a continuously differentiable function such that $v-\psi$ attains a maximum at $(x,q,p,s)$.
		Also, without loss of generality, assume that $v(x,q,p,s)=\psi(x,q,p,s)$.
		Hence,
		\begin{eqnarray}
		\nonumber v(x,q,p,s) &\leq& v(y^u_{x,q,p;s}(s-h),q,p-h,s-h) \\
		\nonumber 				 &\leq& \psi(y^u_{x,q,p;s}(s-h),q,p-h,s-h) 	
		\end{eqnarray}
		and by dividing by $h$ and taking $h \rightarrow 0$ one obtains
		$$\partial_s \psi + \partial_p \psi + f(s,x,u) \cdot \nabla_x \psi \leq 0.$$
		Since $u$ is arbitrary and admissible, we conclude that $A\leq 0$, which completes the proof.
	\end{proof}


	Theorem \ref{HJprop} provides a convenient way to characterize the value function whose level-set is the reachable set defined in \eqref{reachForward}.
	However, in order to be sure that the solution that stems from \eqref{HJ1}-\eqref{HJ3} corresponds to \eqref{valuePenal}, a uniqueness result is necessary.
	This is achieved by a comparison principle which is stated in the next theorem.
	Let $BUSC(\Omega)$ and $BLSC(\Omega)$ respectively be the space of u.s.c. and l.s.c. functions defined over the set $\Omega$.
	
	\begin{theorem}
		Let $u_1 \in BUSC(\Omega)$ and $u_2 \in BLSC(\Omega)$ be, respectively, sub- and supersolution of 
		\begin{eqnarray}
			\partial_s u + \partial_p u + H(s,x,q,\nabla_x u)  \bigwedge u-\varphi(x) ) = 0	 &,& \forall (x,q,p,s) \in \Omega, \label{HJBinsideNUM}\\
			u(x,q,0,s) -(Mu)(x,q,0,s)= 0 							&,& \forall (x,q,s) \in X \times Q \times [0,\infty[	 \\			
			u(x,q,p,0) = \max(\phi(x),\varphi(x))						&,& \forall (x,q,p) \in X \times Q \times P.
		\label{HJBmainNUM}
		\end{eqnarray}
		
		Then, $u_1 \leq u_2$ in $\overline \Omega$.
		\label{compPrinc}
	\end{theorem}

	The proof is inspired by earlier work on uniqueness results for hybrid control problems.
	The idea is to show that $u_1 \leq u_2$ in all domain $\Omega$ and then on the boundary $p=0$. 
	The main difficulty arises when dealing with points in the boundary $p=0$ where the system has a switching condition given by a non-local switch operator.
	This is tackled by the utilization of \lq\lq{}friendly giant\rq\rq{}-like test functions \cite{BarlesSheetal}, \cite{Ley}. 
	Classically, these functions are used to prove uniqueness for elliptic problems with unbounded value functions where they serve to localize some arguments regardless of the function\rq{}s possible growth at infinity.
	This feature proves itself very useful in our case because one can properly split the domain in no-switching and switching regions.
	In this work, the lag condition for the switch serves as an equivalent to the \lq\lq{}landing condition\rq\rq{}-- which states that after an autonomous switch the system must land at some positive distance away from the autonomous switch set \cite{Dharmatti2}, \cite{Zhang}.

\begin{proof}	
	Let $\Omega$ be defined as above, $\partial \Omega|_T =  X  \times Q \times P \times \{0\}$ and $\partial \Omega|_P =  X  \times Q \times \{0\} \times T$. 

	First, the comparison principle is proved for $\partial \Omega|_T$ (case $1$), followed by $\Omega$ (case $2$) and finally for $\partial \Omega|_P$ (case $3$), which concludes the proof for $\closure{\, \Omega}$.

Case $1$: 
	At a point $(x,q,p,t) \in  \partial \Omega|_T$, from the sub- and supersolution properties,
	\begin{eqnarray*}
		u_1(x,q,p,0) - \max(\phi(x),\varphi(x)) &\leq& 0, \\
		-u_2(x,q,p,0) + \max(\phi(x),\varphi(x)) &\leq& 0,
	\end{eqnarray*}
	which readily yields $u_1 \leq u_2$ in $ \partial \Omega|_T$.
	
Case $2$:
	Start by using to sub- and supersolution properties of $u_1,u_2$ to obtain, in $\Omega$,
	\begin{eqnarray}
		\min(\partial_s u_1 + \partial_p u_1 + H(s,x,q,\nabla_x u_1)  , u_1-\varphi(x) ) &\leq& 0, \label{subVisc}\\
		\min(\partial_s u_2 + \partial_p u_2 + H(s,x,q,\nabla_x u_2)  , u_2-\varphi(x) ) &\geq& 0. \label{superVisc}
	\end{eqnarray}
	Expression \eqref{superVisc} implies that both 
	\begin{equation}
		u_2\geq \varphi(x)
		\label{secondHSide}
	\end{equation}
	and
	\begin{equation}
		\partial_s u_2 + \partial_p u_2 + H(s,x,q,\nabla_xu_2) \geq 0.
		\label{firstHSide}
	\end{equation}
	From \eqref{subVisc}, one has to consider two possibilities.
	The first one is when $u_1 \leq \varphi(x)$.
	If so, together with \eqref{secondHSide}, one has immediately $u_1 \leq u_2$.
	Now, if $\partial_s u_1 + \partial_p u_1 + H(s,x,q,\nabla_x u_1) \leq 0$, one turns to \eqref{firstHSide}.

	Define $v = u_1 - u_2$.
	Notice that $v \in BUSC(\Omega)$.
	The next step is to show that $v$ is a subsolution of
	\begin{equation}
		\partial_s v + \partial_p v + H(s,x,q,\nabla_x v) = 0
		\label{psiHJB}
	\end{equation}
	at $(\bar x, \bar q, \bar p, \bar s) $.	
	
	Let $\psi \in C^2(\Omega)$, bounded, be such that $v-\psi$ has a strict local maximum at $(\bar x, \bar q, \bar p, \bar s) \in \Omega$. 
	Define auxiliary functions over $\Omega^i \times \Omega^i$, $i=0,1$ as 
	\begin{eqnarray}
		 \Phi_\epsilon^i (x,p,s,\xi,\pi,\varsigma)& =& u_1(x,i,p,s) - u_2(\xi, i, \pi, \varsigma) - \psi(x,p,s) \label{auxEpsilon}\\
	\nonumber&&  ~~~~~~~~~~~~- \frac{|x - \xi|^2}{2\epsilon} -  \frac{|p - \pi|^2}{2\epsilon} -  \frac{|s - \varsigma|^2}{2\epsilon}.
	\end{eqnarray}
	
	Because the boundedness of $\psi$, $u_1$ and $u_2$ the suprema points are finite, for each $i=0,1$. Denote $(\supeps) \in \Omega^{\bar q} \times \Omega^{\bar q}$ a point such that 
	\begin{equation*}
		\Phi_\epsilon^{\bar q}(\supeps) = \sup_{ \Omega^{\bar q} \times \Omega^{\bar q}}   \Phi_\epsilon^{\bar q} (x,p,s,\xi,\pi,\varsigma).
	\end{equation*}

	The following lemma establishes some estimations needed further in the proof:

	\begin{lemma}
		Define $\Phi_\epsilon^i$ and $(\supeps)$ as above. 
		Then, as $\epsilon \rightarrow 0$,
		\begin{eqnarray*}
			\frac{|x_\epsilon - \xi_\epsilon|^2}{\epsilon} \rightarrow 0, ~~\frac{|p_\epsilon - \pi_\epsilon|^2}{\epsilon} \rightarrow 0, ~~\frac{|s_\epsilon - \varsigma_\epsilon|^2}{\epsilon} \rightarrow 0, \label{lemmaResultA}\\
 					|x_\epsilon - \xi_\epsilon| \rightarrow 0, ~~|p_\epsilon - \pi_\epsilon|\rightarrow 0, ~~|s_\epsilon - \varsigma_\epsilon| \rightarrow 0,
 \label{lemmaResultB}
		\end{eqnarray*}	
	and $(\supeps) \rightarrow (\bar x,\bar p,\bar s,\bar x,\bar p,\bar s)$
		\label{lemmaConver}
	\end{lemma}

	\begin{proof}
		Writing 
		$$2\Phi_\epsilon^i(\supeps) \geq \Phi_\epsilon^i(x_\epsilon,p_\epsilon,s_\epsilon,x_\epsilon,p_\epsilon,s_\epsilon) + \Phi_\epsilon^i(\xi_\epsilon,\pi_\epsilon,\varsigma_\epsilon,\xi_\epsilon,\pi_\epsilon,\varsigma_\epsilon),$$
		for $i=0,1$, one obtains,
		\begin{eqnarray*}
		\frac{|x - \xi|^2}{\epsilon} +  \frac{|p - \pi|^2}{\epsilon} +  \frac{|s - \varsigma|^2}{\epsilon} &\leq& (u_1 + u_2)(x_\epsilon,i,p_\epsilon,s_\epsilon) - (u_1 + u_2)(\xi_\epsilon,i,\pi_\epsilon,\varsigma_\epsilon) + \\
		&&  ~~~~~~~~~~~~~~~~~~~~~~~~~~~~~~~~~~   \psi(x_\epsilon,p_\epsilon,s_\epsilon) -  \psi(\xi_\epsilon,\pi_\epsilon,\varsigma_\epsilon) ,
		\end{eqnarray*}•
		which means that, since $\psi$, $u_1$ and $u_2$ are bounded that	
		\begin{equation}
			\frac{|x_\epsilon - \xi_\epsilon|^2}{\epsilon}  \leq C*, ~~\frac{|p_\epsilon - \pi_\epsilon|^2}{\epsilon}  \leq C*, ~~\frac{|s_\epsilon - \varsigma_\epsilon|^2}{\epsilon}  \leq C^*,
		\label{lemmaInegEps}
		\end{equation}•
		where $C^*$ depends on the $\sup |u_1|$, $\sup |u_2|$, $\sup |\psi|$ and is independent of $\epsilon$.
		Expression \eqref{lemmaInegEps} yields 
		\begin{equation*}
			|x_\epsilon - \xi_\epsilon| \leq \sqrt{\epsilon C^*}, ~~|p_\epsilon - \pi_\epsilon| \leq  \sqrt{\epsilon C^*}, ~~|s_\epsilon - \varsigma_\epsilon| \leq  \sqrt{\epsilon C^*}.
		\end{equation*}•
		which implies that the doubled terms tend to zero.
		
		Since $ (\bar x,\bar q,\bar p,\bar s)$ is a strict maximum of $v-\psi$, one gets  $(\supeps) \rightarrow (\bar x,\bar p,\bar s,\bar x,\bar p,\bar s)$.
		Remark that, since $\bar p >0$, one can always choose a suitable subsequence $\epsilon_n \rightarrow 0$ such that all $p_{\epsilon_n} >0$, avoiding thus touching the switching boundary.	
	\end{proof}
	
	A straightforward calculation allows to show that there exists $a,b \in \RRR$ such that
	\begin{eqnarray*}
	(a,b,D_\epsilon) &\in& D^- u_2(\xi_\epsilon,\bar q,\pi_\epsilon,\varsigma_\epsilon) \\
	(a + \partial_s \psi ,b + \partial_p \psi,D_\epsilon + \nabla_x \psi) &\in& D^+ u_1(x\epsilon,\bar q,p_\epsilon,s_\epsilon) ,
	\end{eqnarray*}
	where $D^-,D^+$ respectively denote the sub- and super differential \cite{Crandall92user’sguide} and $D_\epsilon= 2|x_\epsilon - \xi_\epsilon| / \epsilon$, which implies
	\begin{eqnarray*}
	a+b+ H(\varsigma_\epsilon,\xi_\epsilon,\bar q,D_\epsilon) &\geq& 0 \\
	a + \partial_s \psi + b + \partial_p \psi+ H(s_\epsilon,x_\epsilon,\bar q,D_\epsilon  + \nabla_x \psi) &\leq& 0,
	\end{eqnarray*}
	which in turn yields, as $\epsilon \rightarrow 0$,
	\begin{equation*}
	\partial_s \psi+ \partial_p \psi - L_f | \nabla_x \psi| \leq 0
	\end{equation*}
	at $(\bar x, \bar q, \bar p, \bar s) \in \Omega$.
	By adequately choosing the test functions $\psi$, one can repeat the arguments to show that this assertion holds for any point in $\Omega$.
	Thus, this establishes that $v$ is a subsolution of \eqref{psiHJB} in $\Omega$.
	
	Now, take $\kappa >0$ and define a non-decreasing differentiable function $\chi_\kappa : (-\infty,0) \rightarrow \RRR^+$ such that
	\begin{equation*}
		\chi_\kappa(x) = 0 ,~ x \leq - \kappa~;~\chi_\kappa(x) \rightarrow \infty, ~ x\rightarrow 0.
	\end{equation*}
	
	Take $\eta >0$,and define a test function 
	\begin{equation*}
		\nu(x,p,s) = \eta s^2  + \chi_\kappa(-p).
	\end{equation*}
	Observe that $v - \nu$ achieves a maximum at a finite point $(x_0,\bar q,p_0,s_0) \in \Omega$.
	Since $\kappa$ can be made arbitrarily small one can consider $p_0 > \kappa$ without loss of generality.
	Therefore, using the subsolution property of $v$, by a straightforward  calculation one has
	\begin{equation*}
		2 \eta s_0 \leq 0,
	\end{equation*}
	since $\chi_\eta\rq{}(-p_0) = 0$.
	The above inequality implies that $s_0 = 0$.
	Noticing that $\nu(x_0,p_0,s_0) = v(x_0,\bar q,p_0,s_0) = 0$, it follows
	\begin{equation*}
		v(x,\bar q,p,s) \leq  \eta s^2 + \chi_\kappa(-p) 
	\end{equation*}
	for all $s\in T$, $x\in X$ and $p > \kappa$.
	Letting $\eta \rightarrow 0$, $\kappa \rightarrow 0$ and from the arbitrariness of $\bar q$, we conclude that $v\leq 0$ in $\Omega$.

	Case $3$:
	In this case the switch lock variable arrives at the boundary of the domain, incurring thus a switch, as all others variables remain inside the domain. 
	For all $(x_0,q_0,p_0,s_0) \in \partial \Omega|_P$, for any $p \geq \delta$ one has (using case $2$ and noticing that $M^+u_1 = Mu_1$ and $M^- u_2 = Mu_2$)
	\begin{equation*}
		 (M^+ u_1)(x_0,q_0,p_0,s_0) \leq u_1(x_0,q_0,p,s_0) \leq u_2(x_0,q_0,p,s_0).
	\end{equation*}•
	Taking the infimum with respect to $p$, the above expression yields $M^+u_1 \leq M^-u_2$ in $\partial \Omega|_P$.
	This suffices to conclude, since that by the sub- and supersolution properties
	\begin{equation*}
	v = u_1 - u_2 \leq M^+u_1 - M^-u_2.
	\end{equation*}
\end{proof}

	\section{Numerical Analysis}

		\subsection{Numerical Scheme and Convergence}

	Equations \eqref{HJ1}-\eqref{HJ3} can be solved using a finite differences scheme. 
	This section proposes a class of discretization schemes and shows its convergence using the Barles-Souganidis \cite{souganidis} framework.

	Set mesh sizes $\Delta x >0$, $\Delta p >0$, $\Delta t >0$ and denote the discrete grid point by $(x_{I},p_k,s_n)$, where $x_I = I\Delta x$, $p_k = k \Delta p$ and $s_n = n\Delta t$, with $I \in \ZZZ^d$ and $k,n$ integers. 
	The approximation of the value function is denoted
	\begin{equation*}
		 v(x_I,q,p_k,s_n) = \vv{Ik}{q}{n}
	\end{equation*}
	and the penalization functions are denoted $\phi(x_I) = \phi_I$, $\varphi(x_I) = \varphi_I$.	
	Define the following grids:
	\begin{eqnarray*}
		G^\# &=& I\Delta x \times Q \times \Delta p\{0,1,\cdots,n_p\} \times \Delta t \{0,1,\cdots,n_s \}, \\
		G^\#_H &=& \Delta t \{0,1,\cdots,n_s \} \times I\Delta x \times Q 
	\end{eqnarray*}
	and the discrete space gradient at point $x_I$ for any general function $\mu$:
	\begin{equation*}
	D^\pm \mu(X_I) = D^\pm \mu_I = \left( D^\pm_{x_1} \mu_I ,\cdots,  D^\pm_{x_d} \mu_I \right),
	\end{equation*}•
	where
	\begin{equation*}
	D^\pm_{x_j} \mu_I = \pm \frac{\mu_{I^{j,\pm}} - \mu_I }{\Delta x},
	\end{equation*}•
	with
	\begin{equation*}
	I^{j,\pm} = (i_1,\cdots,i_{j-1},i_j \pm 1,\cdots,i_d).
	\end{equation*}•
	
	Define a numerical Hamiltonian $\HH : G_H^\# \times \RRR^d \times \RRR^d \rightarrow \RRR$ destined to be an approximation of $H$.
	We assume that $\HH$ verifies the following hypothesis:
	
	\begin{description}
		\item[(H6)] There exists $L_{H_1},L_{H_2} > 0$ such that, for all $s,x,q \in G_H^\#$ and $A^+,A^-,B^+,B^- \in \RRR^d$,
			\begin{eqnarray*}	
			\nonumber	| \HH(s,x,q, A^+,A^-) - \HH(s,x,q, B^+,B^-) | &\leq& L_{H_1}(||A^+ - B^+|| + ||A^- - B^- || \\
			 ||\HH(s,x,q, A^+,A^-)|| &\leq& L_{H_2}(||A^+ + A^-||).
			\end{eqnarray*}
		\item[(H7)]  The Hamiltonian satisfies the monotonicity condition for all $s,x,q \in G_H^\#$ and almost every $A^+,A^-\in \RRR^d$:
		\begin{equation*}
			\partial_{A_i^+}\HH(s,x,q, A^+,A^-) \leq 0, \text{ and } \partial_{A_i^-}\HH(s,x,q, A^+,A^-) \geq 0.
		\end{equation*}•
	
		\item[(H8)]  There exists $\L_{H_3} >0$ such that for all $s,x,q \in G_H^\#$, $s\rq{},x\rq{},q\rq{} \in T \times X \times Q$ and $A \in \RRR^d$,
		\begin{equation*}
		|\HH(s,x,q,A,A) - H(s\rq{},x\rq{},q\rq{},A)| \leq L_{H_3}(|s-s\rq{} | + ||x-x\rq{} || + |q-q\rq{}|).
		\end{equation*}•		
					
	\end{description}	
		
	Let $\Phi : \Omega \rightarrow \RRR$, $h = (\Delta x, \Delta p, \Delta t)$ and set
	\begin{eqnarray*}
	\nonumber	 S^\Omega _h(x,q,p,s,\lambda;\Phi)& =&  \min \left( \lambda - \varphi_{I} ,  \HH(s,x,q, D^+ \Phi(x,q,p,s),D^- \Phi(x,q,p,s))+ \right. \\ 
	\nonumber &&   ~~~~~~~~~~~~~~~~~~~~~~\left.  \frac{\lambda - \Phi(x,q,p,s)}{\Delta t} + \frac{  \lambda -  \Phi(x,q,p-\Delta p,s+\Delta t) }{\Delta p}  \right).
	\end{eqnarray*}•
	
	Now, consider the following scheme 
	\begin{equation}
	 S _h(x,q,p,s,\lambda;\Phi) =  \left\{ \begin{array}{lll}
	 S^\Omega _h(x,q,p,s,\lambda;\Phi) &\text{if}& (x,q,p,s) \in \Omega \\
	 \lambda - \min_{w\in W(x,q),p\rq{}\geq \delta} \Phi(x,g(w,q),p\rq{},s) &\text{if}& p=0,
	 \end{array} \right.
	 \label{scheme}
	\end{equation}
	along with the following operator
	\begin{equation}
	 \FFF(x,q,p,s,u,\nabla u) =  \left\{ \begin{array}{lll}
  u-\varphi(x) \bigwedge \partial_s u + \partial_p u + H(s,x,q,\nabla_x u)  &\text{if}&  (x,q,p,s) \in \Omega \\
			u(x,q,p,s) -(Mu)(x,q,p,s) &\text{if}&  p=0. \\		
			 \end{array} \right.
	\label{contscheme}
	\end{equation}
		
	\begin{proposition}
		Let $\Phi \in C^\infty_b(\Omega)$.
		Under hypothesis $(H6-H8)$ and  the CFL condition
		\begin{equation}
		\Delta t \left( \frac{1}{\Delta p} + \frac{1}{\Delta x} \sum_{i=1}^d \partial_{A_i^+}\HH +\partial_{A_i^-}\HH \right)  \leq 1
		\label{CFLcond}
		\end{equation}•
		the discretization scheme \eqref{scheme} of \eqref{contscheme} is stable, monotone and consistent.
		
		Moreover, a solution $u_h$ of \eqref{scheme} converges towards the solution $u$ of \eqref{contscheme} as $h \rightarrow 0$.
	\end{proposition}
	
	\begin{proof}
		The proof follows the lines used in the framework of Barles-Souganidis \cite{souganidis}.
		The goal is to show that the numerical scheme solutions\rq{} envelopes
		\begin{eqnarray*}
			\underline u(x\rq{},q\rq{},p\rq{},s\rq{}) = \liminf\limits_{\substack{(x,q,p,s) \to (x\rq{},q\rq{},p\rq{},s\rq{})\\ h \to 0}} u_h(x,q,p,s) \\
			\overline u(x\rq{},s\rq{},p\rq{},s\rq{}) = \limsup\limits_{\substack{(x,q,p,s) \to (x\rq{},q\rq{},p\rq{},s\rq{})\\ h \to 0}} u_h(x,q,p,s),
		\end{eqnarray*}•
		are respectively supersolution and subsolution of \eqref{contscheme}.
		Then, using the comparison principle in theorem \ref{compPrinc}, one obtains $\overline u \leq \underline u$.
		However, since the inverse inequality is immediate (using the definition of limsup and liminf, one gets $u \equiv \underline u = \overline u$ achieving thus the convergence.
		
		Only the subsolution property of $\overline u$ is presented next, the proof of the supersolution property of $\underline u$ being very alike.
		
		Firstly, the proofs shows that \eqref{scheme} is stable, monotone and consistent.
				Observe that $S$ is proportional to $-\Phi$ in the terms outside the Hamiltonian.
		Since fluxes $\HH$ are monotone by hypothesis $(H7)$ (see \cite{shuOsher} for details in monotone Hamiltonian fluxes) whenever the CFL condition \eqref{CFLcond} is satisfied, the monotonicity of $S$ follows.
		The stability is ensured by the boundedness of $\Phi$ and hypothesis $(H6)$.
		Finally, hypothesis $(H8)$ and lemma \ref{Moperator} are used in a straightforward fashion to obtain the consistency properties below:
		\begin{eqnarray*}
			 \limsup\limits_{\substack{(x\rq{},q\rq{},p\rq{},s\rq{}) \to (x,q,p,s)\\ h \to 0}} S _h(x\rq{},q\rq{},p\rq{},s\rq{},\lambda;\Phi) \leq  \FFF^*(x,q,p,s,\Phi,\nabla \Phi)  \\
			 \liminf\limits_{\substack{(x\rq{},q\rq{},p\rq{},s\rq{}) \to (x,q,p,s)\\ h \to 0}} S _h(x\rq{},q\rq{},p\rq{},s\rq{},\lambda;\Phi) \geq  \FFF_*(x,q,p,s,\Phi,\nabla \Phi) 
		\end{eqnarray*}•
		
		Now, choose  $\Phi \in C^\infty_b(\Omega)$ such that $\overline u - \Phi$ has a strict local maximum at $(x_0,q_0,p_0,s_0) \in \closure{\, \Omega}$ (without loss of generality assume $(\overline u - \Phi)(x_0,q_0,p_0,s_0) = 0$).
		
		First, suppose $p_0>0$.
		Then there exists a ball centered in $(x_0,q_0,p_0,s_0)$ of radius $r >0$ such that $\overline u(x,q,p,s) \leq \Phi(x,q,p,s),~\forall (x,q,p,s) \in B((x_0,q_0,p_0,s_0),r) \subset \Omega$.
		Construct sequences $(x_\epsilon,q_\epsilon,p_\epsilon,s_\epsilon) \to (x_0,q_0,p_0,s_0) $ and $h_\epsilon \to 0$ as $\epsilon \to 0$ such that $u_{h_\epsilon}(x_\epsilon,q_\epsilon,p_\epsilon,s_\epsilon)  \to \overline u(x_0,q_0,p_0,s_0)$ and $(x_\epsilon,q_\epsilon,p_\epsilon,s_\epsilon) $ is a maximum of $u_{h_\epsilon} - \Phi$ in $B((x_0,q_0,p_0,s_0),r)$.
		Denote $\zeta_\epsilon = (u_{h_\epsilon} - \Phi)(x_\epsilon,q_\epsilon,p_\epsilon,s_\epsilon)$.
		(Remark that $\zeta_\epsilon \to 0$ as $\epsilon \to 0$).
		
		Then, $u_{h_\epsilon} \leq \Phi + \zeta_\epsilon$ inside the ball and since $S(x_\epsilon,q_\epsilon,p_\epsilon,s_\epsilon,u_{h_\epsilon}(x_\epsilon,q_\epsilon,p_\epsilon,s_\epsilon);u_{h_\epsilon})=0$, by the monotonicity property one obtains
		\begin{equation*}
			S(x_\epsilon,q_\epsilon,p_\epsilon,s_\epsilon,\Phi(x_\epsilon,q_\epsilon,p_\epsilon,s_\epsilon) + \zeta_\epsilon;\Phi + \zeta_\epsilon) \leq 0.
		\end{equation*}•
		Taking the limit (inf) $\epsilon \to 0$ together with the consistency of the scheme, one obtains the desired inequality
		\begin{equation*}
			\FFF_*(x_0,q_0,p_0,s_0,\Phi, \nabla \Phi) \leq 0.
		\end{equation*}•
		
		Suppose now that $p_0=0$. 
		Construct sequences $(x_\epsilon,q_\epsilon,p_0,s_\epsilon) \to (x_0,q_0,p_0,s_0) $ and $h_\epsilon \to 0$ as $\epsilon \to 0$ such that $u_{h_\epsilon}(x_\epsilon,q_\epsilon,p_0,s_\epsilon)  \to \overline u(x_0,q_0,p_0,s_0)$.
		Then,
		\begin{eqnarray*}
			 \liminf\limits_{\substack{x_\epsilon \to x_0 \\ q_\epsilon \to q_0 \\ s_\epsilon \to s_0 \\ h_\epsilon \to 0}} S_h(x_\epsilon,q_\epsilon,p_0,s_\epsilon,\Phi(x_\epsilon,q_\epsilon,p_0,s_\epsilon),\Phi) &=&   \liminf\limits_{\substack{x_\epsilon \to x_0 \\ q_\epsilon \to q_0 \\ s_\epsilon \to s_0 \\ h_\epsilon \to 0}} (\Phi - M \Phi)(x_\epsilon,q_\epsilon,p_0,s_\epsilon) \\
			  &=&  (\Phi - (M\Phi)_*)(x_0,q_0,p_0,s_0).
		\end{eqnarray*}•
		Since each $u_{h_\epsilon}$ is a  solution of \eqref{scheme}, using lemma \ref{Moperator} the above expression yields at the point $(x_0,q_0,p_0,s_0)$:
		\begin{eqnarray*}
		0 &=&  (\Phi - (M\Phi)_*)(x_0,q_0,p_0,s_0) \\
			&\geq&  (\Phi - (M\Phi)^*)(x_0,q_0,p_0,s_0) \\
			 &\geq&  (\Phi - (M^+\Phi))(x_0,q_0,p_0,s_0) \\
			 &=&\FFF_*(x_0,q_0,p_0,s_0,\Phi, \nabla \Phi)
		\end{eqnarray*}•
		achieving the desired inequality.		
	\end{proof}	
	
		\subsection{Numerical Simulations}

	For the numerical simulations, the numerical Hamiltonian $\HH$ is discretized using a monotone  Local Lax-Friedrichs scheme \cite{shuOsher} (where the two components of the gradient are explicit):
	\begin{eqnarray*}
		\HH \left(t,x,q;a^+,a^-,b^+,b^-\right) &=& H \left(t,x,q;\frac{a^++a^-}{2},\frac{b^++b^-}{2}\right) - \\ 
		&& ~~~~~~~~~~~~~~~~~~~~~~~~~~ c_a  \left(\frac{a^+-a^-}{2}\right) - c_b  \left(\frac{b^+-b^-}{2}\right)
	\end{eqnarray*}•
	where $a^\pm = D_i^\pm v$, $b^\pm = D_j^\pm v$ and the constants $c_a,c_b$ are defined as 
	\begin{eqnarray}
		c_a = \sup_{t,x,q,r} | \partial_{r_a} H(t,x,q,r) | \\
		c_b = \sup_{t,x,q,r} | \partial_{r_b} H(t,x,q,r) |.
	\end{eqnarray}•

	Setting $u_h^\# = [\vv{Ik}{q}{n}, \vv{I^{1^\pm} k}{q}{n}, \cdots,  \vv{I^{d^\pm} k}{q}{n}, \vv{Ik - 1}{q}{n}]$, the equation 
	$$S _h( x_{I},q,p_k,s_{n+1},\vv{Ik}{q}{n+1};u_h^\#) =0$$
	allows an explicit expression of $\vv{}{}{n+1}$ as a function of past values $\vv{}{}{n}$:

	\begin{equation}
		\vv{Ik}{q}{n+1} = \left\{ \begin{array}{lll}
		 \varphi_{I} \bigvee \vv{Ik}{q}{n} - \Delta t  \left(  \frac{ \vv{Ik}{q}{n} - \vv{Ik-1}{q}{n} }{\Delta p} + \HH(t_n,x_{I},q,D^- \vv{Ik}{q}{n}, D^+ \vv{Ik}{q}{n})  \right) \\
		  ~~~~~~~~~~~~~~~~~~~~~~~~~~~~~~~~~~~~~~~~~~~~~~~~~~~~~~~~~~~~~~~~~~~~\text{ if }  (I,q,k,n) \in \Omega^\# \\
		 \min_{w\in W(x_I,q), k' \geq \delta/\Delta p} \vv{Ik'}{g(w,q)}{n}~~~~~~~~~~~~~~~~~~~~~\,~~~~~~ \text{ if } k=0.
		 \end{array} \right.
	\label{schemeSimulation}
	\end{equation}

	In order to illustrate the work presented, a simple vehicle model is used in the simulations.
	This model allows an analytic evaluation of it\rq{}s autonomy and is suitable for an a posteriori verification of the results.
	
	The switch dynamics is given by
	 \begin{equation*}
		g(w,q) = |q-w|.
	\end{equation*}
	The energetic dynamical model is given by $ f(u,q) = ( -a_x + qu , -q(a_y+u))$,
	where $a_x,a_y >0$ are constant depletion rates of the battery\rq{}s electric energy and the reservoir\rq{}s fuel (whenever the RE is on), respectively.
	The control domain is taken $U=[0,u_{\text{max}}]$.
	
	Considering this dynamics, an exact autonomy of the system can be evaluated analytically.
	Given initial conditions $(x_0,y_0)$ the shortest time to empty the fuel reservoir is given by $t^* = y_0 / (a_y+u_{\text{max}})$.
	The SOC evaluated at this instant is given by $x(t^*) = x(0) -t^*(a_x - u_{\text{max}})$. 
	If $x(t^*)\leq 0$, it means the fuel cannot be consumed fast enough before the battery is depleted.  
	This condition can be expressed in terms of the parameters of the model as $x_0(a_y+u_{\text{max}})  \leq y_0 (a_x -u_{\text{max}})$.
	In this case, the autonomy is given by $T^0 = x_0/(a_x-u_{\text{max}}) $. If not, the autonomy is given by $T^1 = (x_0 + u_{\text{max}}t^*) / a_x$.

	Simulations are running using $\Delta x = \Delta y = 0.025$, $\Delta p = 0.05$ and $\Delta t$ is calculated using \eqref{CFLcond}.
	Several instances of $x_0,y_0,a_x,a_y,u_{\text{max}}$ are tested, all with a lag of $\delta = 2$, and the theoretical and evaluated autonomy are compared.
	For numerical purposes, the set of initial energies is a ball of radius $2\Delta x$ around $(x_0,y_0)$ and the admissible region is set as $K=[0,1]^2$.
	Remark that \eqref{distance}, \eqref{penal} give a natural value $\tilde L_K$ for the numerical boundary outside set $K$.
	
	The simulated instances use $a_x = 0.1$, $a_y=0.15$, $u_{\text{max}}=0.07$.
	Tests are made using three initial conditions.
	Table \ref{tab:convSimulations} groups the error $\varepsilon =|T^* - s^*|$ between exact autonomy and the autonomy \eqref{hybridAutonomy} evaluated using the scheme \eqref{schemeSimulation}   and the algorithm running times.
	
	Figures \ref{results1}, \ref{results2} and \ref{results3} show the minimum of the value function for all values of $p \in P^\#$ for $q=0,1$ and the corresponding reachable set of the instance $(x_0,y_0)=(0.3,0.8)$ at times $s=2.65$, $s=3.00$, $s=3.75$ respectively.
	
	\begin{table}
	\centering
		\caption{Convergence results and running times.}
		\begin{tabular}{|c|c|c|c|}
		\hline
		$(x_0,y_0)$ &$\Delta x $ & $\varepsilon$  &  CPU\footnotemark[1]  running time(s)  \\ \hline
		\multirow{4}{*}{(0.5,0.5)}	 &$0.05$ & $0.877$	& $34$		\\ 
							 &$0.04$ & $0.613$	& $57$		\\ 
							 &$0.03$ & $0.860$	& $111$ 		\\ 
							 &$0.02$ & $0.326$	& $326$ 		\\ \hline
		 \multirow{4}{*}{(0.3,0.8)} 	&$0.05$ & $1.923$	& $28$		\\ 
							 &$0.04$ & $1.179$	& $46$		\\ 
							 &$0.03$ & $0.578$	& $89  $		\\ 
							 &$0.02$ & $0.038$	& $253$		\\ \hline
		\end{tabular}
		\label{tab:convSimulations}
	\end{table}


		\begin{figure}[!htpb]
		\begin{center}
			\includegraphics[width=1\columnwidth]{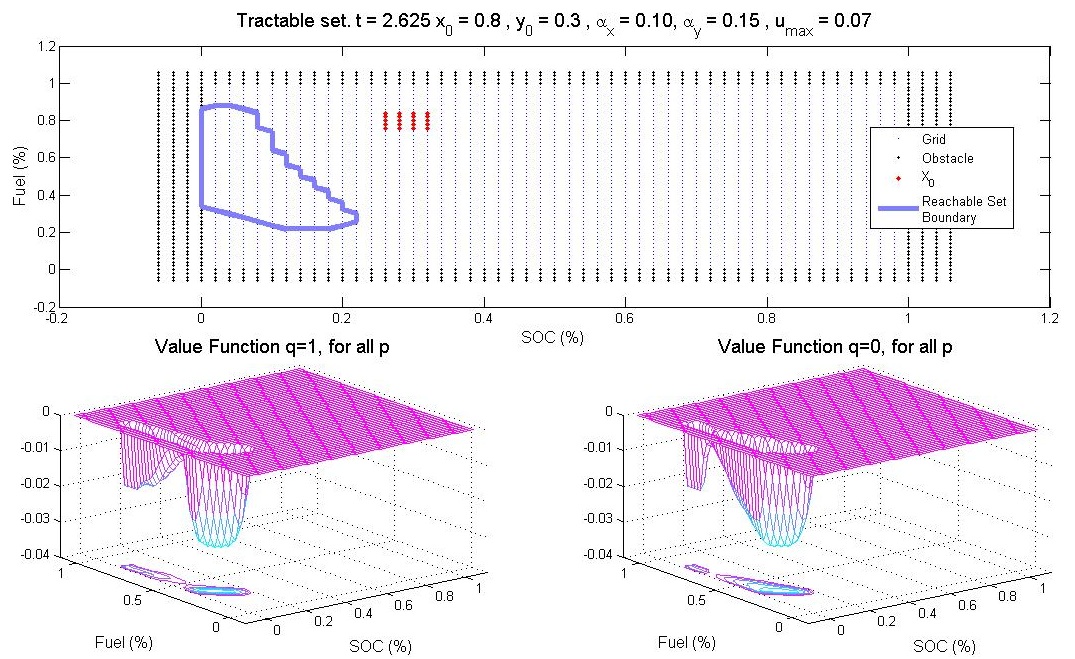} 
			\caption{Reachable set and value functions at $s=2.65$. }
			\label{results1}
		\end{center}
	\end{figure}	
		\footnotetext[1]{Intel Xeon E5504 @ $2\times 2.00$GHz, $2.99$Gb RAM.}

	\begin{figure}[!htpb]
		\begin{center}
			\includegraphics[width=1\columnwidth]{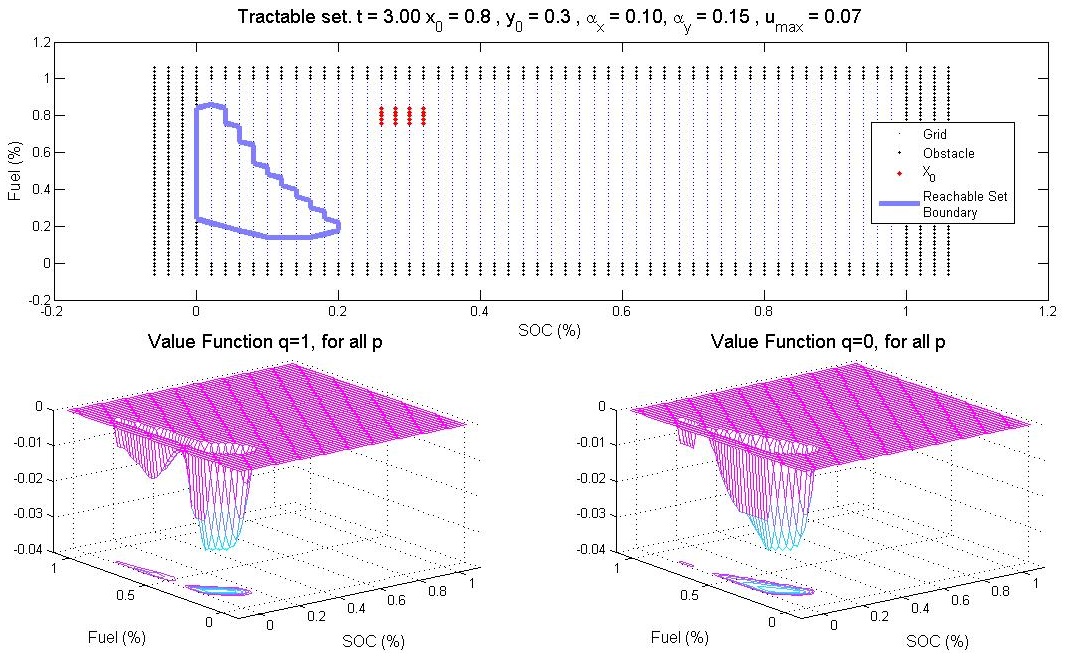} 
			\caption{Reachable set and value functions at $s=3.00$. }
			\label{results2}
		\end{center}
	\end{figure}	
	
	\begin{figure}[!htpb]
		\begin{center}
			\includegraphics[width=1\columnwidth]{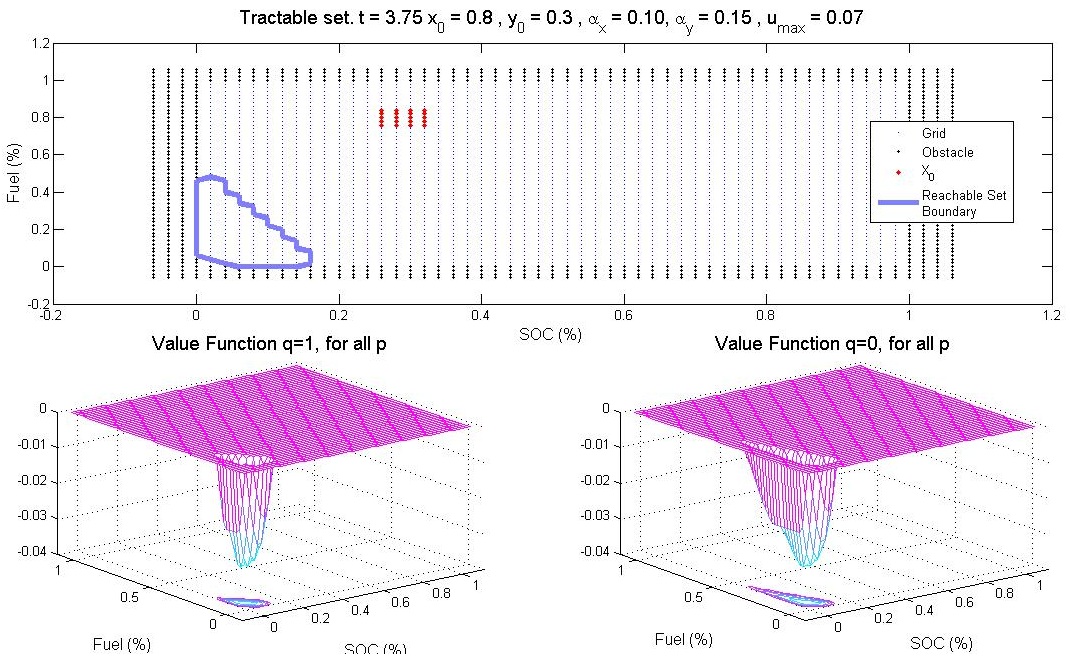} 
			\caption{Reachable set and value functions at $s=3.75$. }
			\label{results3}
		\end{center}
	\end{figure}

\newpage
\clearpage
\bibliography{journal_reachability}

\end{document}